\documentclass[12pt]{article}
\usepackage{latexsym,amsfonts,amsmath,graphics,amssymb}
\usepackage{verbatim}
\usepackage{epsfig}

\newtheorem{theorem}{Theorem}
\newtheorem{lemma}{Lemma}
\newtheorem{corollary}{Corollary}
\newtheorem{proposition}{Proposition}

\newtheorem{definition}{Definition}

\newtheorem{remark}{Remark}
\newenvironment{proof}{\begin{trivlist}
    \item[\hskip\labelsep{\it Proof.}]}{$\hfill\Box$\end{trivlist}}

\newcommand{\satop}[2]{\stackrel{\scriptstyle{#1}}{\scriptstyle{#2}}}

\newcommand{\bsgamma}{\boldsymbol{\gamma}}

\newcommand{\bsc}{\boldsymbol{c}}
\newcommand{\bsk}{\boldsymbol{k}}

\newcommand{\bsw}{\boldsymbol{w}}
\newcommand{\bsx}{\boldsymbol{x}}

\newcommand{\bsg}{\boldsymbol{g}}
\newcommand{\bst}{\boldsymbol{t}}

\newcommand{\bsnu}{\boldsymbol{\nu}}

\newcommand{\bsy}{\boldsymbol{y}}
\newcommand{\bssigma}{\boldsymbol{\sigma}}

\newcommand{\AC}{{\mathcal{A}}}
\newcommand{\wal}{{\rm wal}}

\newcommand{\bszero}{\boldsymbol{0}}
\newcommand{\rd}{\,\mathrm{d}}
\newcommand{\NN}{\mathbb{N}}
\newcommand{\N}{\mathbb{N}}
\newcommand{\ZZ}{\mathbb{Z}}
\newcommand{\integer}{\ZZ}
\newcommand{\FF}{\mathbb{F}}

\newcommand{\R}{\mathbb{R}}

\newcommand{\Hg}{\mathcal{H}_{\bsgamma}}

\renewcommand{\pmod}[1]{\,(\bmod\,#1)}

\newcommand{\e}{\varepsilon}
\newcommand{\X}{{\mathfrak X}}
\newcommand{\U}{{\mathcal U}}
\newcommand{\cS}{\mathcal{S}}
\newcommand{\bsq}{\boldsymbol{q}}

\DeclareMathOperator{\tail}{tail}
\DeclareMathOperator{\decay}{decay}

\DeclareMathOperator{\mo}{mod}

\DeclareMathOperator{\nes}{nest}
\DeclareMathOperator{\unr}{unr}
\DeclareMathOperator{\cost}{cost}
\DeclareMathOperator{\CD}{CD}
\DeclareMathOperator{\ML}{ML}

\begin{document}

\title{\scshape Infinite-Dimensional Integration in Weighted
Hilbert Spaces: Anchored Decompositions, Optimal Deterministic Algorithms,
and Higher Order Convergence}

\author{$\quad$Josef Dick\\$\quad$\\
{\small School of Mathematics and Statistics,
University of New South Wales}\\
{\small Sydney, NSW, 2052, Australia} \\
{\small email: josef.dick@unsw.edu.au}\\
$\quad$\\
Michael Gnewuch\\$\quad$\\
%{\small Institut f\"ur Informatik,} \\
%{\small Christian-Albrechts-Universit\"at zu Kiel} \\
%{\small Christian-Albrechts-Platz 4, 24098 Kiel, Germany} \\
%{\small  email: mig@numerik.uni-kiel.de}}
{\small School of Mathematics and Statistics,
University of New South Wales}\\
{\small Sydney, NSW, 2052, Australia} \\
{\small email: m.gnewuch@unsw.edu.au}}

\date{}
\maketitle

\begin{abstract}
We study numerical integration of functions depending on an infinite number
of variables. We provide lower error bounds for general deterministic linear algorithms
and provide matching upper error bounds with the help of suitable multilevel algorithms
and changing dimension algorithms.

More precisely, the spaces of integrands we consider are weighted reproducing kernel
Hilbert spaces with norms induced by an underlying anchored function space decomposition.
Here the weights model the relative importance of different groups of variables.
The error criterion used is the deterministic worst case error. We study two cost models for function evaluation which depend on
the number of active variables of the chosen sample points, and two classes of weights,
namely product and order-dependent (POD) weights and the newly introduced weights with
finite active dimension. We show for these classes of weights that multilevel algorithms achieve the optimal rate of convergence  in the first cost model while changing dimension
algorithms achieve the optimal convergence rate in the second model.

As an illustrative example, we discuss the anchored Sobolev space with smoothness
parameter $\alpha$ and provide new optimal quasi-Monte Carlo multilevel algorithms
and quasi-Monte Carlo changing dimension algorithms based on higher-order
polynomial lattice rules.
\end{abstract}


\section{Introduction}

The evaluation of
integrals over functions with an unbounded or even infinite number of variables is
an important task
in physics, quantum chemistry or in quantitative finance, see, e.g.,
\cite{Gil08a, WW96} and the references therein. In recent years a large number of
researchers contributed to the design of new algorithms as, e.g., multilevel and changing dimension algorithms or dimension-wise quadrature methods, to approximate such integrals efficiently.
Multilevel algorithms were introduced by Heinrich and Sindambiwe \cite{Hei98, HS99} in the context of integral equations and parametric integration, and by Giles
\cite{Gil08a, Gil08b} in the context of stochastic differential equations.
Changing dimension algorithms were introduced by Kuo et al. \cite{KSWW10} in the context of infinite-dimensional integration in weighted Hilbert spaces and dimension-wise quadrature methods were introduced by Griebel and Holtz \cite{GH10} for multivariate integration. (Changing dimension algorithms and dimension-wise quadrature methods are based on a similar idea.)

In this paper we want to study infinite-dimensional numerical integration on a weighted reproducing kernel Hilbert space of functions with infinitely many variables as it has
been done in
\cite{HW01, KSWW10, NH09, HMNR10, NHMR11, Gne10, PW11, Gne12, B10, BG12, Gne12a}. The Hilbert spaces we consider here posses so-called
anchored function space decompositions.
For a motivation of this specific function space setting and connections to problems in the theory of stochastic processes and mathematical finance we refer to \cite{HMNR10, NH09, NHMR11}.

We provide error bounds for the worst case error of deterministic linear algorithms; these bounds are expressed in terms
of the cost of the algorithms.
We solely take account of function evaluations, i.e., the
cost of function sampling, and neglect other cost as, e.g., combinatorial cost.
To evaluate the cost of sampling, we
consider two cost models: the \emph{nested subspace sampling model} (introduced in \cite{CDMR09}, where it was called \emph{variable subspace sampling model}) and the \emph{unrestricted subspace sampling model}
(introduced in \cite{KSWW10}).

In the nested subspace sampling model lower error bounds for infinite-dimensional
integration were provided in \cite{NHMR11} for general $n$-point quadrature formulas in the case where
the weighted Hilbert space of integrands is defined via an anchored kernel and the weights are product weights. We generalize these error bounds to general weights.
In the unrestricted subspace sampling model lower error bounds where provided
for product weights and anchored kernels in \cite{KSWW10}, and for general weights and
the Wiener kernel in \cite{Gne10}. We generalize these results to anchored kernels and  general weights. (Let us mention that in the randomized setting similar general lower
error bounds for
infinite-dimensional integration on weighted Hilbert spaces are provided
for anchored decompositions in \cite{Gne12} and for underlying ANOVA-type decompositions
in \cite{BG12}; to treat the latter decompositions, a technically more involved analysis
is necessary.)

In this paper we further study two classes of weights in more depths: The class
of \emph{product and order-dependent (POD) weights}, which includes, in particular,
product weights and finite-product weights, and the class of \emph{weights of finite active dimension}, which includes, in particular, finite-diameter weights and (the more general)
finite-intersection weights. We derive several new results for both classes of weights
which might also be of interest for other tractability studies of continuous numerical
problems on weighted spaces, apart from the infinite-dimensional integration problem.

For these two classes of weights we provide upper error bounds with the help of multilevel algorithms and
changing dimension algorithms. These bounds show that for the cost functions most relevant
in applications, namely those cost functions which grow at least linearly in the number
of active variables, the convergence rate of our algorithms is arbitrarily close to the convergence rate of the $N$th minimal integration error and our lower bounds are thus sharp. For the remaining cost functions, which grow sub-linearly in the number of active variables, our bounds are still sharp in most of the cases (depending on the smoothness of the kernel and the decay rate of the weights).

These new upper bounds improve on the results obtained for product weights in \cite{NHMR11} and \cite{Gne10}. Furthermore, in contrast to \cite[Thm.~3]{NHMR11}, we are
able to formulate our results on upper bounds without introducing additional auxiliary weights that are not problem inherent.

We provide explicit quasi-Monte Carlo multilevel and changing dimension algorithms based
on higher order polynomial lattice rules for weighted Hilbert spaces of integrands that
correspond to anchored Sobolev spaces with smoothness parameter $\alpha >1$.
These algorithms are optimal in the sense that they achieve convergence rates arbitrarily
close to the optimal convergence rate (i.e., the convergence rate of the $N$th minimal
integration error).

The article is organized as follows: In Section \ref{TGS} the setting we want to study is
introduced. In Section \ref{LB} we provide lower error bounds for deterministic quadrature formulas for solving the infinite-dimensional integration problem on weighted Hilbert spaces.
In Section \ref{LB_GEN} we present the most general form of the lower bounds which is valid
for arbitrary weights. In Section \ref{LB_SPEC} we state the form of the lower
bounds for the two specific classes of weights we consider.
In Section \ref{MLA} and \ref{CDA} we explain multilevel and changing dimension
algorithms. In Section \ref{UB_POD} we provide upper error bounds for POD weights,
and in Section \ref{UB_FAD} for weights with finite active dimension.
In Section \ref{HOC} we illustrate the upper and lower bounds in the situation where
the space of integrands is based on the univariate anchored Sobolev space with smoothness parameter $\alpha >1$. Here we consider specific quasi-Monte Carlo multilevel and changing dimension
algorithms that achieve higher-order convergence.

\section{The general setting}
\label{TGS}

\subsection{Notation}

For $n\in\N$ we denote the set $\{1,\ldots,n\}$ by $[n]$.
%If $u$ is a finite subset of an infinite set $U$, then we denote this by
%$u\subset_f U$.
If $u$ is a finite set, then its size is denoted by $|u|$.
We put
\begin{equation*}
 \U := \{ u\subset \N \,|\, |u| < \infty \}.
\end{equation*}
We use the common Landau $O$-notation. For two non-negative functions $f$ and $g$
we write occasionally $f=\Omega(g)$ for $g=O(f)$, and $f=\Theta(g)$ if $f=\Omega(g)$
and $f=O(g)$ holds.

\subsection{The function spaces}

As spaces of integrands of infinitely many variables, we consider
\emph{reproducing kernel Hilbert spaces} which are discussed in
more detail in
\cite{HMNR10, GMR12}.
%as in \cite[Sect.~5]{KSWW10} or \cite{PW11}.
Our standard reference for
%these spaces is \cite{GMR12} and, for
general reproducing kernel Hilbert spaces is \cite{Aro50}.

We start with univariate functions. Let $D\subseteq \R$ be a Borel
measurable set of
$\R$ and let $K:D\times D\to \R$ be a measurable reproducing kernel
with \emph{anchor}
$c\in D$, i.e., $K(c,c) = 0$. This implies
$K(\cdot,c) \equiv 0$. We assume that $K$ is non-trivial,
i.e., $K\neq 0$.
We denote the reproducing kernel Hilbert space with kernel $K$ by
$H = H(K)$ and its scalar product and norm by $\langle \cdot, \cdot \rangle_H$ and
$\|\cdot\|_H$, respectively. We use corresponding notation for other
reproducing kernel Hilbert spaces. If $g$ is a constant function in $H(K)$,
then the reproducing property implies
$g = g(c) = \langle g, K(\cdot,c) \rangle_H = 0$.

Let $\rho$ be a probability measure on $D$. We assume that
\begin{equation}
\label{cond-M}
 M:= \int_D K(x,x) \,\rho({\rm d}x) <\infty.
\end{equation}
For arbitrary $\bsx,\bsy \in D^\N$ and $u\in \U$ we define
\begin{equation*}
K_u(\bsx,\bsy) := \prod_{j\in u} K(x_j, y_j),
\end{equation*}
where by convention $K_{\emptyset} \equiv 1$.
The Hilbert space with reproducing kernel $K_u$ will be denoted
by $H_u = H(K_u)$.
Its functions depend only on the coordinates $j\in u$. If it is convenient for us,
we identify $H_u$ with the space of functions defined on $D^u$ determined by
the kernel $\prod_{j\in u} K(x_j, y_j)$, and write $f_u(\bsx_u)$ instead of $f_u(\bsx)$
for $f_u\in H_u$ and $\bsx\in D^{\N}$, where $\bsx_u := (x_j)_{j\in u}\in D^u$.  For all $f_u\in H_u$ and $\bsx\in D^{\N}$ we have
\begin{equation}
\label{vanish}
f_u(\bsx) = 0
\hspace{3ex}\text{if $x_j=c$ for some $j\in u$.}
\end{equation}
This property yields an \emph{anchored decomposition} of functions,
see, e.g., \cite{KSWW10a}.

Let now $\bsgamma = (\gamma_u)_{u\in \U}$ be weights, i.e.,
a family of non-negative numbers. We assume that  $\bsgamma$ satisfies
\begin{equation}
 \label{summable}
\sum_{u\in\U} \gamma_u M^{|u|} <\infty.
\end{equation}
(One may also consider slightly weaker conditions as done, e.g., in \cite[Sect.~5]{KSWW10}
or \cite{PW11}; for a comparison of these different conditions see \cite{GMR12}.)
We denote the \emph{set of active
coordinate sets}, $\{u\in \U \,|\, \gamma_u>0\}$ by $\AC = \AC(\bsgamma)$. (Sets $u \subseteq \mathbb{N}$ with $|u| = \infty$ are always assumed to be inactive.) We always assume that $\AC$ is non-trivial, i.e., that there
exists a $\emptyset \neq u \in \U$ with $u\in \AC$.

Let us define the domain $\X$ of functions of infinitely many variables by
\begin{equation*}
\X:= \left\{\bsx \in D^{\N} \,|\, \sum_{u\in \AC} \gamma_u \prod_{j\in u}
K(x_j,x_j) <\infty \right\}.
\end{equation*}
Let $\mu$ be the infinite-product probability measure of $\rho$ on $D^{\N}$.
Due to our assumptions we have $\mu(\X) = 1$, see \cite[Lemma~1]{HMNR10}
or \cite{GMR12}.
For $\bsx,\bsy \in\X$ we define
\begin{equation*}
\mathcal{K}_{\bsgamma}(\bsx,\bsy) := \sum_{u\in\AC} \gamma_u K_u(\bsx,\bsy).
\end{equation*}
$\mathcal{K}_{\bsgamma}$ is well-defined and, since $\mathcal{K}_{\bsgamma}$ is symmetric and positive
semi-definite, it is a reproducing
kernel on $\X \times \X$, see \cite{Aro50}. We denote the corresponding
reproducing kernel Hilbert space by $\mathcal{H}_{\bsgamma} = H(\mathcal{K}_{\bsgamma})$ and its norm by $\|\cdot\|_{\bsgamma}$.
For the next lemma see \cite[Cor.~5]{HW01} or \cite{GMR12}.

\begin{lemma}
\label{Lemma6}
The space $\Hg$ consists of all functions
$f=\sum_{u\in\AC} f_u$, $f_u\in H_u$, such that
\begin{equation*}
\sum_{u\in\AC} \gamma^{-1}_u \|f_u\|^2_{H_u} <\infty.
\end{equation*}
In the case of convergence, we have
\begin{equation*}
\|f\|^2_{\bsgamma} = \sum_{u\in\AC} \gamma^{-1}_u \|f_u\|^2_{H_u}.
\end{equation*}
\end{lemma}

For $u\in \AC$ let $P_u$ denote the orthogonal projection
$P_u:\Hg \to H_u$, $f\mapsto f_u$ onto $H_u$. Then each $f\in \Hg$
has a unique representation
\begin{equation*}
f = \sum_{u\in \AC} f_u
\hspace{2ex}\text{with $f_u = P_u(f) \in H_u$, $u\in \AC$.}
\end{equation*}

\subsection{Infinite-dimensional integration}
\label{IDI}

Due to (\ref{summable}), we have $\Hg \subseteq L_1(\X,{\rm d}\mu)$, and
the integration functional
\begin{equation*}
I(f) := \int_{\X} f(\bsx)\,\mu({\rm d}\bsx)
\end{equation*}
is continuous on
$\Hg$, i.e., the operator norm of $I$ is finite:
\begin{equation}
\label{bedingung}
\|I\|^2_{{\mathcal{H}_{\bsgamma}}} = \sum_{u\in \AC} \gamma_u C^{|u|}_0 <\infty,
\hspace{2ex}\text{where}\hspace{2ex}
C_0:= \int_D\int_D
K(x,y)\,\rho({\rm d}x)\,\rho({\rm d}y) < \infty,
\end{equation}
see, e.g., \cite{GMR12}.
We assume that $I$ is non-trivial, i.e., that $C_0>0$.
Notice that $C_0 \le M$.

For a given set of weights $\bsgamma$ we denote by $\widehat{\bsgamma}$
the set of weights defined by
\begin{equation}
\label{gammahut}
\widehat{\gamma}_u := \gamma_u C_0^{|u|}
\hspace{2ex}\text{for all $u\in \U$.}
\end{equation}
The representer $h\in \Hg$ of $I$, i.e., the function $h$ satisfying
$I(f) = \langle f, h \rangle_{\bsgamma}$ for all $f\in\Hg$, is given by
\begin{equation*}
h(\bsx) = \int_{\X} \mathcal{K}_{\bsgamma}(\bsx,\bsy) \mu({\rm d}\bsy)
\end{equation*}
and consequently the operator norm of the functional $I$ satisfies
%\begin{equation*}
$\|I\|_{\Hg} = \|h\|_{\bsgamma}$.
%= \left( \sum_{u\in\AC} \widehat{\gamma}_u \right)^{1/2}.
%\end{equation*}
For $u\in\AC$ we define $I_u := I\circ P_u$ on $\Hg$,
i.e., $I_u(f) = \langle f, P_u(h) \rangle_{\bsgamma}$  for all $f\in\Hg$ .
More concretely, we have
\begin{equation*}
I_u(f) = \int_{D^u} f_u(\bsx_u)  \,\rho^u({\rm d}\bsx_u),
\end{equation*}
and the representer $h_u$ of $I_u$ in $\Hg$ is given by
$h_u(\bsx_u) = P_u(h)(\bsx_u)$.
Thus we have
\begin{equation*}
I(f) = \sum_{u\in\AC} I_u(f_u)
\hspace{2ex}\text{for all $f\in\Hg$.}
\end{equation*}

\subsection{Admissible algorithms, errors, and cost models}

We define the \emph{set of admissible sample points} $S$ by
\begin{equation}
\label{admissable_set}
S:= \{ (\bsx_u;\bsc) \,|\, u\in \U\}.
\end{equation}
Here again $\bsx_u = (x_j)_{j\in u}\in D^u$, and
$(\bsx_u;\bsc)$ denotes the vector $\bsy = (y_1,y_2,\ldots) \in D^{\N}$
with $y_j = x_j$ if $j\in u$ and $y_j = c$ otherwise.
Note that $(\bsx_u;\bsc) \in \X$.
We consider algorithms of the form
\begin{equation}
\label{def-alg}
Q(f) = \sum_{i=1}^n a_i f(\bst_{v_i};\bsc),
\hspace{2ex}\text{for $v_1, \ldots, v_n\in \U$,}
\end{equation}
with points $\bst_{v_i}\in (D\setminus \{c\})^{v_i}$ and coefficients $a_i\in\R$.
The worst case error is given by
%\begin{equation*}
%e^{\ran}(Q; \Hg) := \sup_{\|f\|_{\bsgamma} \le 1} \left(
%\EE \big( |I(f)-Q(f)|^2 \big) \right)^{1/2}
%\end{equation*}
%in the randomized setting, and
\begin{equation*}
e(Q; \Hg) := \sup_{\|f\|_{\bsgamma} \le 1}
|I(f)-Q(f)|.
\end{equation*}
%Notice that if the algorithm $Q$ is deterministic, then
%$e^{\ran}(Q;\Hg)$
%reduces to the worst case error $e^{\wor}(Q;\Hg)$.
For an algorithm $Q$ of the form (\ref{def-alg}) we put
$(Q)_{u} := Q\circ P_u$, i.e.,
\begin{equation*}
(Q)_{u}(f) = \sum^n_{i=1} a_i f_u(\bst_{v_i\cap u}; \bsc).
\end{equation*}
We have the identity
\begin{equation}
\label{worid}
[e(Q;\Hg)]^2 =
\sum_{u\in\AC} \gamma_u [e((Q)_{u}; H_u)]^2,
\end{equation}
where
\begin{equation*}
e((Q)_{u}; H_u) = \sup_{\|g\|_{H_u} \le 1}
|I_u(g)- (Q)_{u}(g)|.
\end{equation*}

For the cost of an algorithm we only take into account the cost
for function evaluations.
% i.e., the cost of an algorithm $Q$ which
%is of the form (\ref{def-alg}) (and which is actually executed in
%that form), is in principle
%\begin{equation*}
%\cost(Q) = \sum^n_{i=1} \cost(L_{(\bst_{v_i};\bsc)}),
%\end{equation*}
%where $L_{\bsx}$ denotes again the point evaluation in $\bsx$.
%To quantify the cost of a function evaluation
To make this more precise, let us fix a
\emph{cost function} $\$:\N\to [1,\infty)$, which is non-decreasing.
In this paper we consider two models for the cost of function
evaluations,
%the fixed subspace sampling,
the nested subspace
sampling and the unrestricted subspace sampling model.

%In the \emph{fixed subspace sampling model}, we first fix a
%coordinate set
%$v$ (which will often be of the form $v=[d] := \{1,\ldots, d\}$ for a fixed $d\in %\N$),
%depending on the required approximation accuracy, and allow only
%sample points of the form $(\bsx_{v};\bsc)$. Each function evaluation
%then costs $\$(|v|)$.

In the \emph{nested subspace sampling model} we first define for a fixed
strictly increasing sequence $\bsw = (w_i)_{i\in\N}$ of coordinate sets
$w_1\subset w_2 \subset \cdots \in \U$ the cost of a function evaluation
in $\bsx\in\X$ to be
\begin{equation}
 \label{sechs}
\mathfrak{c}_{\bsw,c}(\bsx) := \inf\{\$(|w_i|) \,|\, x_j=c \hspace{1ex}\forall j\notin w_i\}.
\end{equation}
Here we use the standard convention that $\inf\emptyset = \infty$.
For a linear algorithm $Q$ of the form (\ref{def-alg}) we define
\begin{equation*}
\mathfrak{c}_{\bsw,c}(Q) := \sum^n_{i=1} \mathfrak{c}_{\bsw,c}(\bst_{v_i};\bsc).
\end{equation*}
Let $C_{\nes}$ denote the set of all cost functions $c_{\bsw, c}$ of the form \eqref{sechs} where $\bsw$ runs through all strictly increasing sequences $\bsw$ of coordinate sets. Then we define the cost of $Q$ in the nested subspace sampling model to be
\begin{equation*}
 \cost_{\nes}(Q) := \inf_{\mathfrak{c}_{\bsw,c}\in C_{\nes}} \mathfrak{c}_{\bsw,c}(Q).
\end{equation*}
This model was introduced in
\cite{CDMR09}.\footnote{In \cite{CDMR09} it was actually called ``variable subspace sampling
model''. We have chosen a different name to emphasize the difference between
this model and the ``unrestricted subspace sampling model''
explained below.}

In the \emph{unrestricted subspace sampling model} a function
evaluation $f(\bsx)$ costs
\begin{equation*}
% \label{sechs}
\mathfrak{c}_{c}(\bsx) := \inf\{\$(|u|) \,|\, u\in \U\,, \hspace{1ex}
x_j=c \hspace{1ex}\forall j\notin u\}.
\end{equation*}
The cost of a linear algorithm $Q$ of the form (\ref{def-alg}) in the unrestricted
subspace sampling model is given by
\begin{equation*}
\cost_{\unr}(Q) := \sum^n_{i=1} \mathfrak{c}_{c}(\bst_{v_i};\bsc) = \sum^n_{i=1} \$(|v_i|).
\end{equation*}
The \emph{unrestricted subspace sampling model} was introduced in
\cite{KSWW10}.\footnote{In \cite{KSWW10} the cost model did not get a specific name.}

We denote the cost of an algorithm $Q$ in the nested and
unrestricted subspace sampling model by
$\cost_{\nes}(Q)$ and $\cost_{\unr}(Q)$,
respectively.
Obviously, the unrestricted subspace sampling model is more generous
than the nested subspace sampling model.
%The fixed subspace sampling model is clearly the most expensive one.
Note that in both sampling models the cost for function evaluations in
non-admissible sample points is infinite.

\subsection{Strong tractability}
\label{tractabilitysection}

Let $\mo \in \{ \nes, \unr\}$.
The {\em $\e$-complexity} is defined as the minimal
cost among all algorithms of the form (\ref{def-alg}), whose worst case
errors are at most $\e$, i.e.,
\begin{equation}
\label{comp}
  {\rm comp}_{\mo}(\e;\Hg)
 \,:=\, \inf\left\{{\rm cost_{\mo}}(Q) \,|\,
Q \hspace{1ex}\text{is of the form (\ref{def-alg}) and}
\hspace{1ex} e(Q;\Hg)\le\e\right\}.
\end{equation}
The integration problem $I$ is said to be {\em strongly
tractable}\footnote{We chose this notion, since it seems to us to be consistent with the usual notion of tractability in the multivariate setting. A more precise notion would be
``strongly polynomially tractable'', to distinguish this kind of tractability from more
general notions of tractability as introduced in \cite{GW07}, see also \cite{NW08}. But for
convenience we stay with the shorter notion ``strongly tractable''.}
if there are non-negative constants $C$ and $p$ such that
\begin{equation}
\label{pol-tr}
    {\rm comp}_{\mo}(\e;\Hg)\le C \,\e^{-p} \qquad
   \mbox{for all $\e>0$}.
\end{equation}
The {\em exponent of strong tractability} is given by
\begin{equation*}
p^{\mo}= p^{\mo}(\bsgamma) := \inf\{ p\,|\, \text{$p$ satisfies \eqref{pol-tr}}\}.
\end{equation*}
Essentially, $1/p^{\mo}$ is the
\emph{convergence rate} of the
\emph{$N$th minimal worst case error}
\begin{equation}
\label{worstcaseerr}
e^{\mo}(N;\Hg) := \inf\{ e(Q;\Hg )\,|\,
Q \hspace{1ex}\text{is of the form (\ref{def-alg}) and}
\hspace{1ex} \cost_{\mo}(Q)\le N\}.
\end{equation}
In particular, we have for all $p>p^{\mo}$ that
%\begin{equation}
$e^{\mo}(N;\Hg) = O(N^{-1/p})$.
%\end{equation}

\subsection{Weights}
\label{WEIGHTS}

Here we introduce further definitions and notation which is necessary
for our analysis of lower and upper bounds for the exponents of
strong tractability in the different models.

Let $\bsgamma=(\gamma_u)_{u\in \U}$
be a given family of weights.
Weights $\bsgamma$ are called \emph{finite-order weights
of order $\omega$}
if  there exists an
$\omega\in \N$ such that
$\gamma_{u}=0$ for all $u\in \U$ with $|u|>\omega$.
Finite-order weights were introduced in \cite{DSWW06} for spaces
of functions with a finite number of variables.
The following definition is taken from \cite{Gne10}.

\begin{definition}
\label{Cut-Off}
For weights $\bsgamma$ and $\sigma \in\N$ let us define the \emph{cut-off weights}
of order $\sigma$
\begin{equation}
\label{gammasigma}
\bsgamma^{(\sigma)}
= (\gamma_u^{(\sigma)})_{u\in \U}
\hspace{2ex}\text{via}\hspace{2ex}
\gamma^{(\sigma)}_{u} =
\begin{cases}
\,\gamma_u
\hspace{2ex}&\text{if $|u| \le \sigma$},\\
\,0
\hspace{2ex} &\text{otherwise.}
\end{cases}
\end{equation}
\end{definition}

Clearly, cut-off weights of order $\sigma$ are in particular
finite-order weights of order $\sigma$.

We always assume that the weights $\bsgamma$ we consider satisfy
(\ref{summable}).
%As already mentioned before, we assume that we have weights
%$\bsgamma$ satisfying
%\begin{equation*}
%\label{summable}
%\sum_{u\in \U} \widehat{\gamma}_u < \infty;
%\end{equation*}
%this ensures that the
%integration functional $I$ is continuous on $\Hg$.

Let us denote by $u_1(\sigma), u_2(\sigma),\ldots $,
the distinct non-empty sets $u\in \U$ with
$\gamma_u^{(\sigma)} >0$
for which
$\widehat{\gamma}_{u_1(\sigma)}^{(\sigma)} \ge
\widehat{\gamma}_{u_2(\sigma)}^{(\sigma)} \ge \cdots$.
Let us put $u_0(\sigma) := \emptyset$. We can make the
same definitions for
$\sigma = \infty$; then we have obviously
$\bsgamma^{(\infty)} = \bsgamma$.
For convenience we will usually suppress any reference to $\sigma$
in the case where $\sigma = \infty$.
For $\sigma\in\N\cup\{\infty\}$ let us define
\begin{equation*}
\tail_{\bsgamma,\sigma} (d):= \sum_{j=d+1}^\infty
\widehat{\gamma}_{u_j(\sigma)}^{(\sigma)} \in [0,\infty]
\hspace{2ex}\text{and}\hspace{2ex}
\decay_{\bsgamma,\sigma} :=
\sup \left\{ p\in \R \,\Big|\, \lim_{j\to\infty}
\widehat{\gamma}_{u_j(\sigma)}^{(\sigma)}j^p =0
\right\}.
\end{equation*}

The following definition is from \cite{Gne10}.

\begin{definition}
%\label{Tsternsigma}
For $\sigma\in\N\cup\{\infty\}$ let $t^*_\sigma \in [0,\infty]$
be defined as
\begin{equation*}
t^*_\sigma := \inf \big\{t\ge 0\,|\, \,\exists\, C_t>0 \,\,\forall
\,v \in \U: |\{i\in \N \,|\,
u_i(\sigma) \subseteq v\}| \le C_t|v|^t \big\}.
\end{equation*}
%In the case where $\sigma =\omega$, we will use the shorthand
%$t^* = t^*_\omega$.
\end{definition}

Let $\sigma\in\N$. Since $|u_i(\sigma)|\le \sigma$ for all $i\in\N$, we have
obviously $t^*_\sigma \le \sigma$.
On the other hand, if we have an infinite sequence
$(u_j(\sigma))_{j\in\N}$, it is not hard to verify that
$t^*_\sigma \ge 1$, see \cite{Gne10}.

In the following two subsections we describe the classes of weights we
want to consider in this article.

%\subsubsection{Product weights and finite product weights}
\subsubsection{Product and order-dependent weights}

\emph{Product and order-dependent (POD) weights} $\bsgamma$ were introduced in
\cite{KSS11} and are a hybrid of so-called \emph{product weights} and
\emph{order-dependent weights}.
Their general form is
\begin{equation}
\label{pod}
\gamma_u = \Gamma_{|u|} \prod_{j\in u}\gamma_j,
\hspace{3ex}\text{where $\gamma_1\ge\gamma_2\ge \cdots \ge 0$, and
$\Gamma_0=\Gamma_1=1$, $\Gamma_2,\Gamma_3,\ldots \ge 0$.}
\end{equation}
Special cases are product and finite-product weights that are defined as
follows.

\begin{definition}
Let $(\gamma_j)_{j\in\N}$ be a sequence of non-negative real
numbers satisfying $\gamma_1\ge \gamma_2 \ge \ldots.$ With the
help of this sequence we define for $\omega\in\N\cup\{\infty\}$  weights
$\bsgamma = (\gamma_{u})_{u\subset_f\N}$ by
\begin{equation}
\label{gammafpw}
\gamma_{u} =
\begin{cases}
\prod_{j\in u} \gamma_j
\hspace{2ex}&\text{if $|u| \le \omega$},\\
\,0
\hspace{2ex} &\text{otherwise,}
\end{cases}
\end{equation}
where we use the convention that the empty product is $1$.
In the case where $\omega=\infty$, we call such weights \emph{product weights},
in the case where $\omega$ is finite, we call them \emph{finite-product weights of order}
(at most) $\omega$.
\end{definition}

Product weights were introduced by  Sloan and Wo\'zniakowski in \cite{SW98}
and have been studied extensively since then.
Finite-product weights were considered in \cite{Gne10} and are
obviously finite-order weights of order at most $\omega$.

It is easily seen that product weights and finite product weights of order $\omega$ are POD weights;
in (\ref{pod}) one just has to choose $\Gamma_\nu =1$ for all $\nu\in\N$ to obtain
product weights and $\Gamma_{|u|} = 1$ for $|u| \le \omega$
and $\Gamma_{|u|} = 0$ for $|u| > \omega$ to obtain finite product weights.
Other concrete examples of POD weights can be found in \cite{KSS11, KSS12}.

\subsubsection{Algorithmic dimension}\label{sec_alg_dim}

The following definition introduces the concept of the \emph{algorithmic dimension} of
a family of weights.

\begin{definition}
Let $\mathcal{W}\subseteq \U$. Let $d \in \mathbb{N} \cup \{\infty\}$ be such that there exists a function
\begin{equation}
\label{phi}
\phi: \mathbb{N} \to [d]
\hspace{2ex}\text{with the property}\hspace{2ex}
\forall u \in \mathcal{W}\,\, \forall j
\neq j' \in u:\, \phi(j) \neq \phi(j'),
\end{equation}
where $[\infty] = \mathbb{N}$. That is, $\phi|_u$ is injective for each $u\in\mathcal{W}$. If $d \in \mathbb{N}$, then we say that $\mathcal{W}$ has
\emph{finite algorithmic dimension}. In this case we call the minimal $d^\ast = d^\ast(\mathcal{W})$ for which such a $\phi$ exists the \emph{algorithmic dimension} of $\mathcal{W}$.

Let $\bsgamma = (\gamma_u)_{u\in \U}$ be a family of weights.
If its set $\mathcal{A}$ of active coordinate sets has algorithmic
dimension $d^\ast(\mathcal{A})$, we say that the family of weights $\bsgamma$ has
\emph{algorithmic dimension $d^\ast(\bsgamma) := d^\ast(\mathcal{A})$}.
If we do not want to specify the
algorithmic dimension $d^\ast$, we just say that $\bsgamma$ has
\emph{finite algorithmic dimension}.
\end{definition}

Weights $\bsgamma$ of finite algorithmic dimension $d^*$ are obviously
finite-order weights of order $\omega \le d^*$, but finite-order weights do
not necessarily have finite algorithmic dimension.

We define a graph associated with $\mathcal{W}$ in the following way. For a given set $\mathcal{W} \subseteq \U$ we consider the infinite simple graph $G_{\mathcal{W}} =(\N,E_{\mathcal{W}})$, where $(i,j)$ with $i \neq j$, belongs to the set of edges $E_{\mathcal{W}}$ if and only if there exists
a $u\in \mathcal{W}$ with $i,j \in u$. The graph $G_{\mathcal{W}}$ does not contain loops, i.e. edges $(i,i)$. We call $G_{\mathcal{W}}$ the \emph{associated graph}
of $\mathcal{W}$. Notice that two different subsets $\mathcal{W}$, $\mathcal{W}'$ of
$\U$ may have the same associated graph.

The following lemma connects the concept of minimal algorithmic dimension to the \emph{chromatic number} $\chi(G_{\mathcal{W}})$ of $G_{\mathcal{W}}$. Recall that the chromatic number of a graph $G$ is the minimal number of colors needed to color the vertices of $G$ in such a way that any two vertices connected by an edge have a different color.

\begin{lemma}\label{lem_graph}
Let $\mathcal{W} \subseteq \U$ and $G_{\mathcal{W}}$ be the associated graph. Then the algorithmic dimension $d^\ast(\mathcal{W})$ of $\mathcal{W}$ and the chromatic number $\chi(G_{\mathcal{W}})$ coincide, i.e.
\begin{equation*}
d^\ast(\mathcal{W}) = \chi(G_{\mathcal{W}}).
\end{equation*}
\end{lemma}

\begin{proof}
Assume that we have given a coloring of the vertices of the graph $G_{\mathcal{W}}$. Let the vertices of $G_{\mathcal{W}}$ be denoted by $\mathbb{N}$ and the colors be denoted by $1,2,\ldots, \chi(G_{\mathcal{W}})$. Then we can define the function $\phi:\mathbb{N} \to [\chi(G_{\mathcal{W}})]$ by setting $\phi(i) = c_i$, where $c_i \in [\chi(G_{\mathcal{W}})]$ denotes the color of the vertex $i$. On the other hand, if we have a function $\phi:\mathbb{N} \to [d^\ast(\mathcal{W})]$ given, then we can obtain a coloring of the graph $G_{\mathcal{W}}$ by coloring the vertex $i$ by $\phi(i)$. By the definition of the function $\phi$ and the graph $G_{\mathcal{W}}$ this yields a coloring of the graph $G_{\mathcal{W}}$. Since both $d^\ast(\mathcal{W})$ and $\chi(G_{\mathcal{W}})$ are minimal, the result follows.
\end{proof}

With the help of Lemma \ref{lem_graph} we derive in the following remark a lower bound on the algorithmic dimension.

\begin{remark}
A complete graph $G$ with $n$ vertices has chromatic number $n$, since all vertices are connected to each other by an edge and hence all vertices must have a different color.
If $\mathcal{W}$ has algorithmic dimension $d\in\N$, then $|u| \le d$ for all coordinate sets $u$ in $\mathcal{W}$, since $G_{\mathcal{W}}$ contains a subgraph which is a complete graph with $|u|$ vertices. Hence
\begin{equation}\label{lowboualgdim}
d^*(\mathcal{W}) \ge \sup_{u\in\mathcal{W}}|u|.
\end{equation}
Thus weights with algorithmic dimension
$d\in\N$ are necessarily finite-order weights of order $\omega \le d$.

The lower bound (\ref{lowboualgdim}) is not necessarily sharp, as
shown by the following example: Let $|u| \le 2$ for all $u \in \mathcal{W}$ and let there exist a sequence of sets $\{i_1,i_2\}, \{i_2,i_3\}, \ldots, \{i_{k-1}, i_k\}$,  $\{i_k,i_1\} \in \mathcal{W}$ where $k$ is odd. In other words, $G_{\mathcal{W}}$ contains an odd cycle. Then this graph has chromatic number $3$ as can easily
be shown. An even more drastic example is the set $\mathcal{W}:= \{ u\in\U \,|\, |u|=2\}$,
which has not even finite algorithmic dimension.
\end{remark}

Let us now turn to upper bounds on the algorithmic dimension.

\begin{remark}
As a consequence of Lemma \ref{lem_graph}, we obtain that if $G_{\mathcal{W}}$ is a planar graph (meaning that every finite subgraph is planar), then the famous Four Color Theorem
\cite{AH77a, AH77b}
says that $G_{\mathcal{W}}$ can be colored with at most four colors. Hence in
this situation the minimal algorithmic dimension of $\mathcal{W}$ is at most four.
\end{remark}

We provide further upper bounds on the algorithmic dimension in Theorem \ref{bound_dstar} and \ref{Brooks}.

\begin{theorem}\label{bound_dstar}
Let $\mathcal{W} \subseteq \U$. Then the minimal algorithmic dimension of $\mathcal{W}$ is bounded by
\begin{equation*}
d^\ast(\mathcal{W}) \le \sup_{i \in \mathbb{N}} \left|\bigcup_{u \in \mathcal{W}: i \in u} u \right|.
\end{equation*}
\end{theorem}

\begin{proof}
By Lemma~\ref{lem_graph} it follows that it suffices to show that $\chi(G_{\mathcal{W}})$ satisfies the bound. By \cite[Theorem~8.1.3]{Diestel} it follows that $\chi(G_{\mathcal{W}})$ is equal to the maximum of the chromatic numbers $\chi(H)$ over all finite subgraphs $H$ of $G_{\mathcal{W}}$. Thus it suffices to show that for all finite subgraphs $H$ of $G_\mathcal{W}$ the chromatic number $\chi(H)$ satisfies the bound.

Let $H$ be an arbitrary finite subgraph of $G_{\mathcal{W}}$ and let $V_H$ denote the set of vertices of $H$. By \cite[p.115]{Diestel} we have $\chi(H) \le \Delta(H)+1$, where $\Delta(H)$ is the maximum degree of the vertices of $H$. But the degree of a vertex $i$ in the graph $G_{\mathcal{W}}$ is equal to
\begin{equation*}
\Delta(i) = \left|\bigcup_{u \in \mathcal{W}: i \in u} u \right| - 1.
\end{equation*}
By taking the maximum of the degrees over all vertices in the graph $G_{\mathcal{W}}$ we obtain the result.
\end{proof}

In some circumstances the above result can be slightly improved using Brooks' theorem from graph theory, see \cite[Theorem~8.1.3]{Diestel}.

\begin{theorem}\label{Brooks}
Let $\mathcal{W}\subseteq \U$ such that $\sup_{u \in \mathcal{W}} |u| \ge 3$. Let $Z =  \sup_{i \in \mathbb{N}} \left|\bigcup_{u \in \mathcal{W}: i \in u} u \right|$. Let $i_1,i_2,\ldots$ be the set of vertices for which $|\cup_{u \in \mathcal{W}: i_k \in u} u| = Z$. Assume that for each $k \ge 1$ the subgraph consisting of the vertices in $\cup_{u \in \mathcal{W}: i_k \in u} u$ is not complete. Then
\begin{equation*}
d^\ast(\mathcal{W}) \le \max_{i \in \mathbb{N}} \left|\bigcup_{u \in \mathcal{W}: i \in u} u \right| -1.
\end{equation*}
\end{theorem}

Various other bounds on $d^\ast$ can be obtained from graph theory via bounds on the chromatic number of the associated graph, see for instance \cite{Diestel}.

\begin{remark}
In general it is difficult to find a function $\phi$ as in (\ref{phi})
for a given set $\mathcal{W}$. This can be done by a greedy algorithm for graph coloring, see \cite[p. 114]{Diestel}. However, this algorithm does not necessarily find a coloring with the smallest possible number of colors.
\end{remark}

A particular class of weights whose set $\mathcal{W} = \AC$ of active coordinate sets has a finite minimal algorithmic dimension $d$, is the class of finite-intersection weights defined in \cite{Gne10}.

\begin{definition}
\label{def-fiw}
Let $\rho \in \N$.
The finite-order weights $(\gamma_{u_i})_{i\in \N}$, where $\gamma_{u_i} > 0$, are
called \emph{finite-intersection weights} with \emph{intersection
degree} at most $\rho\in \N_0$ if we have
\begin{equation}
\label{fiw}
|\{j\in\N \, | \, u_i\cap u_j \neq \emptyset \}| \le 1+\rho
\hspace{2ex}\text{for all $i\in\N$.}
\end{equation}
\end{definition}

Note that for finite-order weights condition (\ref{fiw}) is
equivalent to the following
condition: There exists an $\eta\in\N$ such that
\begin{equation}
\label{cond}
|\{ i\in\N \,|\, k\in u_i \}| \le \eta
\hspace{2ex}\text{for all $k\in\N$.}
\end{equation}
Indeed, if (\ref{fiw}) is satisfied, then (\ref{cond}) holds with
$\eta \le 1+\rho$, and if (\ref{cond}) is satisfied, then
(\ref{fiw}) holds with $\rho \le (\eta-1)\omega$.

Due to \cite[Lemma~3.10]{Gne10} the set $\AC$ of active
coordinate sets of finite intersection
weights has algorithmic dimension $d^\ast(\AC)$ at most
$[\eta(\omega -1)+1]$; this was shown by constructing inductively a
mapping $\phi:\N \to [\eta(\omega -1)+1]$ that satisfies (\ref{phi}). It also follows from Theorem~\ref{bound_dstar} by
\begin{equation*}
d^\ast(\mathcal{A}) \le \max_{i \in \mathbb{N}} \left|\bigcup_{u \in \mathcal{A}: i \in u} u\right| \le \max_{i \in \mathbb{N}} |\{ u \in \mathcal{A}\,|\, i \in u\} | (\omega-1) + 1 \le \eta (\omega-1) + 1.
\end{equation*}

\section{Lower bounds}
\label{LB}

Here we provide lower bounds for the exponents of tractability in the nested
and in the unrestricted subspace sampling model.
We assume that there exist constants $\varrho,\beta >0$ such that
the $n$th minimal error of univariate integration on $H=H(K)$ satisfies
\begin{equation}
\label{assuni}
e(n;H) \ge \varrho (n+1)^{-\beta}
\hspace{2ex}\text{for all $n\in\N_0$,}
\end{equation}
where
\begin{equation}
e(n;H) := \inf \bigg\{e(Q;H) \,\bigg|\,
Q(f) = \sum^n_{i=1} a_i f(x^{(i)}) \hspace{1ex}
\text{with $a_i\in \R$, $x^{(i)}\in D$} \bigg\}.
\end{equation}

Since for $\emptyset \neq u \in \U$ the integration
problem over $H_u$ is at least as hard as in the univariate
case, assumption (\ref{assuni}) results in
\begin{equation}
\label{betalowbou}
e(Q_{u}; H_u) \ge \varrho C^{\frac{|u|-1}{2}}_0 (n+1)^{-\beta}
\end{equation}
for any quadrature of the form
\begin{equation*}
Q_{u}(f) = \sum_{i=1}^n a_i f(x^{(i)}),
\hspace{2ex} a_i\in\R, x^{(i)}\in D^u, f\in H_u,
\end{equation*}
see \cite[Theorem~17.11]{NW10}.
If now $Q$ is an algorithm of the form (\ref{def-alg}) and
$(Q)_u=Q\circ P_u$, then (\ref{worid}) and (\ref{betalowbou}) imply
\begin{equation}
\label{blowbou}
[e(Q;\Hg)]^2 = \sum^\infty_{j=0} \gamma_{u_j}
[e((Q)_{u_j}; H_{u_j})]^2
\ge b^2 \sum^\infty_{j=1}
\frac{\widehat{\gamma}_{u_j}}{(n_j+1)^{2\beta}},
\end{equation}
where $b^2 := \varrho^2 C_0^{-1}$ and
$n_j := |\{v_i\,|\, u_j\subseteq v_i\}|$.
Since we assumed that $\AC$ is non-trivial, we obtain from (\ref{blowbou})
\begin{equation}
\label{pmodbeta}
p^{\nes} \ge p^{\unr} \ge 1/\beta.
%\hspace{2ex}\text{for all $\mo \in \{\nes, \unr\}$.}
\end{equation}

\subsection{Lower bounds for general weights}
\label{LB_GEN}

In this section we study general weights; here ``general'' means that we only
require the condition (\ref{summable}) to hold.

\subsubsection{Nested subspace sampling}

We start with a new lower bound for the exponent of strong tractability
for general weights in the nested subspace sampling model.

\begin{theorem}
\label{NesLowBou}
Let $\$(k) = \Omega(k^s)$ for some $s> 0$, and let
$\bsgamma$ be weights that satisfy (\ref{summable}).
%Let the integration problem be tractable in the unrestricted subspace
%sampling model, i.e., $p^{\mo} <\infty$.
Then $I$ is only strongly tractable in the nested subspace sampling model if
$\decay_{\bsgamma} >1$. In this case,
\begin{equation}
\label{neslowbou}
p^{\nes} \ge \max \left\{ \frac{1}{\beta}\,,\,
\sup_{\sigma\in\N} \frac{2 s/t^*_{\sigma}}{\decay_{\bsgamma,\sigma} - 1} \right\}.
\end{equation}
\end{theorem}

\begin{proof}
 Let $Q$ be of form (\ref{def-alg}) with $\cost_{\nes}(Q) \le N$.
Then there exists an increasing sequence of sets $\bsw = (w_i)_{i\in\N}$
such that $c_{\bsw,c}(Q) \le N+1$. Let $m$ be the largest integer that satisfies
$\$(|w_m|) \le N+1$. Hence, $v_1,\ldots, v_n \subseteq w_m$. Let $\sigma \in \N$, and let $\bsgamma^{(\sigma)}$ be the corresponding
cut-off weights of $\bsgamma$. Then it is easily seen that
$e(Q; \mathcal{H}_{\bsgamma}) \ge  e(Q; \mathcal{H}_{\bsgamma^{(\sigma)}})$,
cf. \cite[Remark~3.3]{Gne10}. Thus we get from (\ref{blowbou})
\begin{equation*}
[e(Q; \mathcal{H}_{\bsgamma})]^2 \ge b^2  \sum_{j: u_j(\sigma) \nsubseteq
w_m} \widehat{\gamma}^{(\sigma)}_{u_j(\sigma)}.
\end{equation*}
Let now $t>t^*_{\sigma}$. Then, for a suitable constant $C_t>0$,
\begin{equation*}
 \tau_m := |\{j \,|\, u_j(\sigma) \subseteq w_m\}| \le C_t |w_m|^t
= O(N^{t/s}),
\end{equation*}
since $N +1 \ge \$(|w_m|) = \Omega(|w_m|^s)$. Hence we obtain
for every $p_\sigma > \max\{1,\decay_{\bsgamma, \sigma}\}$
\begin{equation*}
[e(Q; \mathcal{H}_{\bsgamma})]^2 \ge b^2  \sum_{j=\tau_m +1}^\infty \widehat{\gamma}^{(\sigma)}_{u_j(\sigma)}
= \Omega(\tau_m^{1-p_\sigma})
= \Omega (N^{t(1-p_\sigma)/s}).
\end{equation*}
This shows that $I$ is only strongly tractable if $\decay_{\bsgamma} >1$.
In that case,
\begin{equation*}
p^{\nes} \ge  \frac{2 s/t^*_{\sigma}}{\decay_{\bsgamma,\sigma} - 1}.
\end{equation*}
From this and (\ref{pmodbeta}) follows the statement of the theorem.
\end{proof}

Note that we have on the one hand $t^*_1 \le t^*_2 \le t^*_3 \le
\cdots$, and on the other hand $\decay_{\bsgamma,1} \ge
\decay_{\bsgamma,2} \ge \decay_{\bsgamma,3} \ge \cdots$.
Thus it is not a priori clear for which $\sigma\in \N$ the supremum
in (\ref{unrlowbou}) is attained. As shown in \cite{Gne10} and as we will
see below,
this may vary for different classes of weights.

\subsubsection{Unrestricted subspace sampling}

The next theorem is a generalization of \cite[Cor.~4.1]{Gne10}, where
only the specific kernel $K(x,y) = \min\{x,y\}$ on $D\times D = [0,1]^2$
was treated.

\begin{theorem}
\label{UnrLowBou}
Let $\$(k) = \Omega(k^s)$ for some $s> 0$, and let
$\bsgamma$ be weights that satisfy (\ref{summable}).
%Let the integration problem be tractable in the unrestricted subspace
%sampling model, i.e., $p^{\mo} <\infty$.
Then $I$ is only strongly tractable in the unrestricted subspace sampling model
if $\decay_{\bsgamma} >1$. In this case,
\begin{equation}
\label{unrlowbou}
p^{\unr} \ge \max \left\{ \frac{1}{\beta}\,,\,
\sup_{\sigma\in\N} \frac{2\min\{ 1,s/t^*_{\sigma} \}}{\decay_{\bsgamma,\sigma} - 1} \right\}.
\end{equation}
\end{theorem}

\begin{proof}
The proof of Theorem \ref{UnrLowBou} is essentially identical with the
one of Theorem 3.4 and Corollary 4.1 in \cite{Gne10}. One just has
to keep in mind that the simple lower bound $p^* \ge 1$ appearing there has to
be replaced by $p^{\unr} \ge 1/\beta$, see (\ref{pmodbeta}).
\end{proof}

\subsection{Lower bounds for special classes of weights}
\label{LB_SPEC}

\subsubsection{Product and order-dependent weights}

Recall that POD weights include as special cases product weights and finite
product weights. We now present a generalized version of \cite[Lemma~3.8]{Gne10},
which holds not only for product and finite product weights, but for general
POD weights.

\begin{lemma}
\label{Lemma3.8}
Let $\bsgamma = (\gamma_u)_{u\in \U}$ be POD weights as in (\ref{pod}). Then
\begin{equation*}
\decay_{\bsgamma,1} = \decay_{\bsgamma,\sigma}
\hspace{3ex}\text{for all $\sigma \in\N$.}
\end{equation*}
This holds still if we replace condition (\ref{summable}) by the weaker condition
that the weights $\widehat{\bsgamma}$ are bounded and have only $0$ as accumulation point.
\end{lemma}

\begin{proof} Let $\sigma\in \N$. Since
$\decay_{\bsgamma,1} \ge \decay_{\bsgamma,\sigma} \ge 0$,
it remains to show that $\decay_{\bsgamma,1} \le \decay_{\bsgamma,\sigma}$.
We can confine ourselves to the case $\decay_{\bsgamma,1}>0$. Let $p\in (0,\decay_{\bsgamma,1})$.
This implies $\sum_{j\in\N} \gamma^{1/p}_j < \infty$. Thus we get
\begin{equation*}
\begin{split}
&\sum_{j\in\N}\widehat{\gamma}^{1/p}_{u_j(\sigma)}
\le \max_{\nu\in [\sigma]} \Gamma^{1/p}_\nu \sum_{j\in\N}\prod_{j\in u_j(\sigma)}
(\gamma_j C_0)^{1/p}
\le \max_{\nu\in [\sigma]} \Gamma^{1/p}_\nu  \prod_{j\in\N} \big(
1 + (\gamma_j C_0)^{1/p} \big)\\
\le &\max_{\nu\in [\sigma]} \Gamma^{1/p}_\nu  \exp \bigg( \sum_{j\in\N} \ln \big(
1 + (\gamma_j C_0)^{1/p} \big) \bigg)
\le \max_{\nu\in [\sigma]} \Gamma^{1/p}_\nu  \exp \bigg( \sum_{j\in\N}
(\gamma_j C_0)^{1/p} \bigg) < \infty,
\end{split}
\end{equation*}
where we used the estimate $\ln(1+x) \le x$, which holds for all non-negative $x$.
Since the sequence $\widehat{\gamma}_{u_j(\sigma)}$, $j\in\N$, is monotonically
decreasing, this implies $\widehat{\gamma}_{u_j(\sigma)} = o(j^{-p})$.
Hence $p\le \decay_{\bsgamma,\sigma}$. Since we may choose $p$ arbitrarily close to
$\decay_{\bsgamma,1}$, we obtain $\decay_{\bsgamma,1} \le \decay_{\bsgamma,\sigma}$.
\end{proof}

For POD weights with $\decay_{\bsgamma}>1$ Lemma \ref{Lemma3.8}, and Theorem \ref{NesLowBou} and
\ref{UnrLowBou} imply strong tractability and
\begin{equation}
\label{neslowboupod}
p^{\nes} \ge \max \left\{ \frac{1}{\beta}, \frac{2 s}{\decay_{\bsgamma,1} - 1}
\right\}
\hspace{2ex}\text{and}\hspace{2ex}
p^{\unr} \ge \max \left\{ \frac{1}{\beta}\,,\,
\frac{2\min\{ 1,s \}}{\decay_{\bsgamma,1} - 1} \right\}.
\end{equation}

For product weights
the lower bound for $p^{\nes}$ can be derived from \cite[Thm.~4]{NHMR11}, and the one for $p^{\unr}$ from \cite[Thm.~3.3 \& Sect.~5.6]{KSWW10}.

Notice that the lower bounds for $p^{\nes}$ and $p^{\unr}$ for finite-product weights are not weaker than for product weights.

\subsubsection{Weights with finite algorithmic dimension}

For the special case of finite-intersection weights of order $\omega$
it was observed in \cite{Gne10}
that if $\AC(\bsgamma^{(\sigma)}) = \infty$, then $t^*_\sigma = 1$ for all $\sigma \in \N$.
Hence for finite-intersection weights the lower bounds (\ref{neslowbou}) and (\ref{unrlowbou}) result in
\begin{equation}
\label{neslowboufad}
p^{\nes} \ge \max \left\{ \frac{1}{\beta}\,,\,
\frac{2s}{\decay_{\bsgamma,\omega} - 1} \right\}
\hspace{2ex}\text{and}\hspace{2ex}
p^{\unr} \ge \max \left\{ \frac{1}{\beta}\,,\,
\frac{2\min\{ 1,s \}}{\decay_{\bsgamma,\omega} - 1} \right\}.
\end{equation}

For the Wiener kernel $K(x,y) = \min\{x,y\}$, defined on $[0,1]^2$, the lower
bound for $p^{\unr}$ in (\ref{neslowboufad}) was already proved in
\cite[Sect.~3.1.1]{Gne10}.

For general weights of finite algorithmic dimension it is however not necessarily  true that $t^*_\sigma = 1$ for all $\sigma\in\N$ as the following two lemmas show.

\begin{lemma}\label{lem_alg_dim1}
%For a given set of weights $\bsgamma = (\gamma_u)_{u \in \U}$ let %$\mathcal{W} = \{u \in \U: \gamma_u > 0\}$. Assume that the 5algorithmic dimension $d^\ast(\mathcal{W}) < \infty$. Then f
Let  $d\in\N$. Then there exists a set of weights $\bsgamma$ with algorithmic dimension $d$ such that for all $k>d$ there exists a $v \in \U$ with $|v| = k$ and
\begin{equation}
\label{vest}
|\{u \subseteq v: u \in \mathcal{\AC(\bsgamma^{(\sigma)})}\}| \ge
\left( \left\lfloor \frac{|v|}{d} \right\rfloor \right)^\sigma {d \choose \sigma} > \left(\frac{|v|}{d}-1\right)^\sigma
\hspace{2ex}\text{for all $\sigma \in [d]$.}
\end{equation}
\end{lemma}

\begin{proof}
%As in Section~\ref{sec_alg_dim} the set $\mathcal{W}$ can be associated with a %graph $\widetilde{G}$. Let us denote the coordinates in $v$ by $1,2,\ldots, |v|$, %i.e., $v = \{1,2,\ldots, |v|\}$.
We construct a graph $\widetilde{G}$ with vertex set $\mathbb{N}$ and chromatic number $d$ in the following way: color the vertex $j \in \mathbb{N}$ by the color $c \in [d]$ given by $c \equiv j \pmod{d}$. Now each pair of vertices $(i,j) \in \mathbb{N}^2$ is an edge of the graph $\widetilde{G}$ if and only if
$i \not\equiv j \pmod{d}$, i.e., if the coloring of the vertices $i$ and $j$ differs. Let $k>d$ and $\sigma \in [d]$ be given. Let $G$ be the subgraph of $\widetilde{G}$ with vertex set $v:= [k]$. Thus for any set $u\subset v$ that consists of $\sigma$ differently colored vertices, the corresponding subgraph is complete.
We now provide a lower bound for the number of ways a subset of $v$ having $\sigma$ differently colored vertices can be chosen. Let $r=\lfloor k/d \rfloor$.
For each color $c \in [d]$, there are at least $r$ vertices in $v$ with color $c$. There are ${d \choose \sigma}$ ways of choosing a set of $\sigma$ different colors out of the $d$ possible colors and for each color $c$ there are at least $r$ possible choices of vertices with this color $c$. Thus the number of possible choices is at least $r^\sigma {d \choose \sigma}$. Hence $G$ contains
at least $r^{\sigma}{d \choose \sigma}$ cliques of size $\sigma$. We now may
define $\bsgamma$, e.g., by $\gamma_u = \prod_{j\in u} j^{-2}$ if $u$ is a clique
in $\widetilde{G}$ and $\gamma_u =0$ else. By construction the algorithmic dimension
of $\bsgamma$ is $d$, see Lemma \ref{lem_graph}, and in addition (\ref{vest}) holds.
\end{proof}

\begin{lemma}
\label{Tsternsigma}
For each $d \in\N$ there exists a set of weights $\bsgamma$ such that $\AC(\bsgamma)$ has algorithmic dimension $d$ and such that for all
$\sigma \in \N \cup \{\infty\}$ we have
%\begin{equation*}
$t^\ast_\sigma = \min\{\sigma, d\}$.
%\end{equation*}
\end{lemma}

\begin{proof}
For $d\in\N$ let $\bsgamma$ be weights as in Lemma \ref{lem_alg_dim1}. Due
to (\ref{vest}) we have for all $\sigma \in \N \cup \{\infty\}$ that
$t^*_{\sigma} \ge \min\{\sigma,d\}$. Since the algorithmic dimension of
$\bsgamma$ is $d$, we have additionally that $t^*_{\sigma} \le d$.
Since always $t^*_{\sigma} \le \sigma$, the statement of the lemma is valid.
\end{proof}

For general weights with finite algorithmic dimension we
just know that the values $\decay_{\bsgamma,1},\ldots,\decay_{\bsgamma,\omega}$
satisfy the relation $\decay_{\bsgamma,1} \ge \ldots \ge \decay_{\bsgamma,\omega}$. We can, e.g., easily construct weights of finite algorithmic dimension
whose set of active coordinate sets $\AC(\bsgamma)$ consists only of sets
of size at least $\sigma \in \{2,\ldots,\omega\}$. Thus $\decay_{\bsgamma,1}
= \ldots = \decay_{\bsgamma, \sigma-1} = \infty$, but $\decay_{\bsgamma, \sigma}$
may be either finite or infinite. Together with Lemma \ref{Tsternsigma} this
argument shows that for general weights with finite algorithmic dimension
we should use the general form of the bounds (\ref{neslowbou}) and (\ref{unrlowbou})
to fully exploit the specific features of the weights we are working with.

\section{Upper bounds}
\label{UB}

Here we provide constructive upper bounds on the exponents of tractability
in the nested and in the unrestricted subspace sampling model. To this purpose
we consider two types of algorithms: \emph{multilevel algorithms}, which perform well
in the nested subspace sampling model, and \emph{changing dimension algorithms},
which are well suited for the unrestricted subspace sampling model.

\subsection{Multilevel algorithms}
\label{MLA}

Let us describe the
general form of the algorithms we want to use more precisely:

Let $L_0:=0$, and let $L_1<L_2<L_3 <\ldots$ be natural numbers,
and let
\begin{equation}
\label{vk12}
v^{(1)}_k := \cup_{j\in [L_k]} u_j
\hspace{2ex}\text{and}\hspace{2ex}
v^{(2)}_k := [L_k]
\hspace{2ex}\text{for $k\in\N$}.
\end{equation}
In the general case we will use the sets
$v^{(1)}_k$, $k=1,\ldots,m$.
In the special cases of POD weights,
it is more convenient to
make use of the relatively simple
ordering of the corresponding set system $u_j$, $j\in\N$,
and choose
the sets $v^{(2)}_k$ for $k=1,\ldots,m$.
In all definitions and results that hold for both choices
of the $v_k^{(i)}$, $i=1,2$, we simply write $v_k$, and we put $v_0:=\emptyset$.
We will choose the numbers $L_1,L_2,\ldots$ in general such that
$|v_k| = \Theta(a^{k})$ for some $a\in (1,\infty)$. (A default choice
would be $a=2$.) Let
%\begin{equation*}
%V_1 := \{j\in\N\,|\, u_j\subseteq v_1\}
%\end{equation*}
%and
\begin{equation*}
V_k := \{j\in\N \,|\,  u_j\subseteq v_k
\hspace{1ex}\text{and}\hspace{1ex} u_j\not\subseteq v_{k-1}\}
\hspace{2ex}\text{for $k\ge 1$.}
\end{equation*}
Let us furthermore define
$$
U(m) := \cup_{k=1}^m V_k \cup \{0\}.
$$

For $u\in \U$ we define the mapping $\Psi_{u}:\Hg\to \Hg$ by
\begin{equation*}
(\Psi_{u}f)(\bsx) = f(\bsx_u;\bsc)
\hspace{2ex}\text{for all $\bsx\in D^\N$.}
\end{equation*}
%Furthermore, we use the shorthand $\Psi_k:= \Psi_{v_k,c}$.

We put
\begin{equation}
\label{algobaustein}
Q_{v_k}(f)
:= \sum_{j=1}^{n_k} a_j^{(k)}
f(\bst^{(j,k)}_{v_{k}};{\bf c}), \hspace{2ex}\text{and}\hspace{2ex} \widehat{Q}_{k}(f) := Q_{v_k}(f - \Psi_{v_{k-1}}f),
\end{equation}
where the numbers $n_1\ge n_2 \ge \ldots \ge n_m$, the coefficients $a_j^{(k)}$,
and the points
$\bst^{(1,k)}_{v_{k}},
\ldots,\bst^{(n_k,k)}_{v_{k}}\in [0,1]^{v_k}$
will be chosen later, depending on the weights $\bsgamma$.

Define the \emph{multilevel algorithm} $Q^{\ML}_m$ via
\begin{equation}
\label{multilevel-algo}
Q^{\ML}_m(f)
:= f(\bsc) + \sum_{k=1}^m \widehat{Q}_{k}(f)
= f(\bsc) + \sum_{k=1}^m \sum_{j=1}^{n_k} a_j^{(k)}
(f - \Psi_{v_{k-1}}f) (\bst^{(k,j)}_{v_{k}};{\bf c}).
\end{equation}

If we choose the nested sequence of coordinate sets $v_1 \subset v_2
\subset v_3 \subset \ldots$ in the nested subspace sampling model,
then the cost of the multilevel algorithm $Q^{\ML}_m$ satisfies
\begin{equation}
\label{costML}
\cost_{\nes}(Q^{\ML}_m) \le \$(0) + 2 \sum^m_{k=1} n_k \$(|v_k|),
\end{equation}
and the same cost bound is valid in the more generous unrestricted
subspace sampling model.
From (\ref{worid}) we obtain
\begin{equation*}
[e(Q^{\ML}_m;\Hg)]^2 = \sum_{j\in\N_0}
\gamma_{u_j} [e((Q^{\ML}_{m})_{u_j};H_{u_j})]^2,
\end{equation*}
where $(Q^{\ML}_{m})_{u_j} = Q^{\ML}_{m} \circ P_{u_j} = \sum^m_{k=1}
(\widehat{Q}_{k})_{u_j}$. Note that
%\begin{equation*}
$e((Q^{\ML}_{m})_{u_0};H_{u_0}) =
e((Q^{\ML}_{m})_{\emptyset};H_{\emptyset}) = 0$,
%\end{equation*}
since $Q^{\ML}_m$ is exact on constant functions. Notice furthermore that
we have $(\widehat{Q}_{k})_{u_j}(f) = 0$ whenever $j\notin V_k$, and
$(\widehat{Q}_{k})_{u_j}(f) = (Q_{v_k})_{u_j}(f) = Q_{v_k}(f_{u_j})$
if $j\in V_k$. Thus we get
\begin{equation}
\label{worMLid}
[e(Q^{\ML}_m;\Hg)]^2 = \sum^m_{k=1} \sum_{j\in V_k}
\gamma_{u_j} [e((Q_{v_k})_{u_j};H_{u_j})]^2
+ \sum_{j\notin U(m)} \widehat{\gamma}_{u_j}.
\end{equation}
Let us now for simplicity assume that $v_k = [\max v_k]$ for all $k\in\N$, which
is always possible by simply renumbering the variables recursively.
Helpful for the construction of good multilevel algorithms for
higher order convergence and general weights is a result of the following kind:

There exists an $\alpha \ge 1/2$ such that for each $k\in \N$ and each $n_k\in\N$
we find a quadrature $Q_{v_k}$ as in (\ref{algobaustein})
which satisfies in the case $\alpha = 1/2$ for $\tau = 1/2$, and in the case
$\alpha >1/2$ for $\tau \in [1/2, \min\{\alpha, \decay_{\bsgamma}/2\})$, $\tau$
arbitrarily close to $\min\{\alpha, \decay_{\bsgamma}/2\}$, the bound
\begin{equation}
\label{assumption3.9}
\sum_{\ell \in u \subseteq [\ell]} \gamma_{u} \left[ e \left( \left(
Q_{v_k} \right)_{u}; H_{u} \right) \right]^2
\le \widehat{C}_{\ell,\tau,\gamma} n_k^{-2\tau}
\hspace{3ex}\text{for all $\ell \in v_k\setminus v_{k-1}$,}
\end{equation}
where
\begin{equation}
\label{cktaugamma}
\widehat{C}_{\ell,\tau,\gamma} = \left( \sum_{\ell \in u \subseteq [\ell]} \gamma^{1/(2\tau)}_{u}
C^{|u|}_\tau \right)^{2\tau}
\hspace{3ex}\text{for some $C_\tau$ independent of $k$.}
\end{equation}

For many reproducing kernels $K$ quadratures like this can be constructed
as quasi-Monte Carlo quadratures. Examples are (shifted) rank-$1$
lattice rules or polynomial lattice rules constructed with the
help of a component-by-component algorithm, see Section \ref{NUMINT}
or, e.g., \cite[Theorem~8]{K}, \cite[Corollary~5.4]{DKPS}, \cite[Prop.~3.9]{Gne10}.

If we use algorithms $Q_{v_k}$ that satisfy condition (\ref{assumption3.9}) to define $\widehat{Q}_k$ as in (\ref{algobaustein}),
then we obtain from (\ref{worMLid})
\begin{equation}
\label{errorestimate}
[e(Q^{\ML}_m;\Hg)]^2 \le \sum^m_{k=1} C_{k,\tau,\gamma} n_k^{-2\tau}
+ \sum_{j\notin U(m)} \widehat{\gamma}_{u_j},
\end{equation}
where
\begin{equation}
\label{c_k_tau_gamma}
C_{k,\tau,\gamma} = \sum_{\ell \in v_k\setminus v_{k-1}} \widehat{C}_{\ell,\tau,\gamma}.
\end{equation}
The aim is to minimize the right hand side of this error bound for
given cost by choosing $\tau$, $m$, and $n_1,\ldots,n_m$
(nearly) optimal. To this purpose one needs a good estimate for the
constants $C_{k,\tau,\gamma}$ and for the tail
$\sum_{j\notin U(m)} \widehat{\gamma}_{u_j}$, i.e., more specific
information about the weights.

\subsection{Changing dimension algorithms}
\label{CDA}

For given weights $\bsgamma$ let $\mathcal{A}_0$ be a finite subset
of $\mathcal{A}(\bsgamma)$.
A \emph{changing dimension algorithm} $Q^{\CD}$ is an algorithm of the
form
\begin{equation}
\label{CD-algo}
Q^{\CD}(f) = \sum_{u \in \mathcal{A}_0} Q_{n_u,u}(f_u),
\end{equation}
where the integrand $f \in \Hg$ has the uniquely determined anchored decomposition
\begin{equation*}
f(\bsx) = \sum_{u\in\mathcal{A}} f_u(\bsx)
\end{equation*}
and $Q_{n_u,u}$ is a quadrature rule for approximating $I_u(f_u)$.
If the building blocks $Q_{n_u,u}$ are linear algorithms, then also
$Q^{\CD}$ is linear; this follows from the explicit formula
\begin{equation*}
f_u(\bsx) = \sum_{v \subseteq u} (-1)^{|u\setminus v|} f(\bsx_v;\bsc)
\end{equation*}
for arbitrary $u\in\AC$, see  \cite{KSWW10a}.
Thus a function evaluation $f_u(\bsx)$ can be done at cost bounded by
$|\{v\in\AC \,|\, v\subseteq u\}| \$(|u|) \le 2^{|u|}\$(|u|)$.
%The main idea here is to find a small subset
%$\AC_0$ of $\AC$ that contains the most important coordinate sets $u$ and
Changing dimension algorithms for infinite-dimension\-al integration were
introduced in \cite{KSWW10}. For POD weights we use a slight modification
of the changing dimension algorithms presented in \cite{PW11} and for weights with
finite active dimension we employ the changing dimension algorithms from \cite[Sect.~4]{KSWW10}.

\subsection{Product and order-dependent weights}
\label{UB_POD}

We consider now product and order-dependent weights (POD) weights, where for each $u \in\U$ we have
\begin{equation*}
\gamma_u = \Gamma_{|u|} \prod_{j \in u} \gamma_j,
\end{equation*}
where $(\Gamma_{|u|})_{u \in \U}$ and $(\gamma_j)_{j\in \mathbb{N}}$ are sequences of nonnegative real numbers as in (\ref{pod}). (Note to distinguish between $\gamma_u$, where $u \in \U$ is a finite set of positive integers, and $\gamma_d, \gamma_j$, where $d, j \in \N$ are positive integers.
%Thus, for instance, $\gamma_{\{1\}}$ is different from $\gamma_1$ if $\Gamma_1 \neq 1$.
)

Before we present the concrete algorithms that we use to obtain upper
bounds for the exponents of tractability $p^{\nes}$ and $p^{\unr}$,
we provide some useful results on POD weights.

\begin{lemma}\label{lem_pod_example}
Let $p^\ast \ge 2q^\ast \ge 2$ such that $p^\ast/(2q^\ast) \in \mathbb{N}$. For the POD weights determined by $\gamma_j = j^{-p^*}$ for $j\in\N$, $\Gamma_0=1=\Gamma_1$, and
\begin{equation*}
\Gamma_k= (k!)^{p^\ast} k^{p^\ast/2-q^\ast} \left(\frac{(p^\ast/q^\ast) \sin ( q^\ast \pi/ p^\ast )}{\pi} \right)^{k p^\ast}
\hspace{2ex}\text{for $k\ge 2$,}
\end{equation*}
we have
\begin{equation*}
\mathrm{decay}_{\bsgamma,\infty} = q^\ast \quad \mbox{ and } \quad \mathrm{decay}_{\bsgamma,\sigma} = p^\ast \quad \mbox{for all } \sigma \in \mathbb{N}.
\end{equation*}
\end{lemma}
A rigorous proof of Lemma \ref{lem_pod_example} can be found in
Section \ref{APPENDIX}.
We suspect that the condition $p^\ast/(2q^\ast) \in \mathbb{N}$ in the above lemma is not necessary. If the condition $q \le p^\ast/2$ can be replaced by $q \le p^\ast$ in Corollary~\ref{cor_pod_criteria} in Section \ref{APPENDIX}, then the condition $p^\ast \ge 2 q^\ast$ can be replaced by $p^\ast \ge q^\ast$ in the above lemma.

Lemma~\ref{lem_pod_example} considers the boundary case where for given product weights $\gamma_j$, the $\Gamma_k$ are made as large as possible such that the POD weights still have finite decay. This allows us to obtain cases where the decay of the POD weights differs from the decay of the corresponding product weights, cf. also Lemma \ref{Lemma3.8}. In the following theorem we consider POD weights where $\Gamma_k$ is smaller such that the decay of the POD weights is always the same as the decay of the corresponding product weights.

\begin{theorem}\label{Corollary2}
Let $\bsgamma = (\gamma_u)_{u\in \U}$ be POD weights with $\gamma_u
= \Gamma_{|u|}\prod_{j\in u} \gamma_j$. Let $p^*:= \decay_{\bsgamma,1} >1$
and $q\le p^*$. Let there exist a constant $C_q>0$ such that
$\Gamma_k \le C_q (k!)^q$ for all $k\in\N$.
In the case where $q=p^*$, we additionally assume
$\sum^\infty_{j=1} \gamma_j^{1/p^*} <1$.
Then we get the following results:

If $p^*=q$, then
\begin{equation*}
 \sum_{d\in u\subseteq [d]} \gamma_u^{1/p^*} = \Theta(\gamma_d^{1/p*}).
\end{equation*}

If $p^* > q$, then
\begin{equation*}
 \sum_{d\in u\subseteq [d]} \gamma_u^{1/p} = \Theta(\gamma_d^{1/p})
\hspace{2ex}\text{for all $p\in (q,p^*)$.}
\end{equation*}
The last identity holds also for $p = p^*$ if $\sum^\infty_{j=1} \gamma_j^{1/p^*} <\infty$.

In particular, our assumptions lead for all $q\le p^*$ to
$\decay_{\bsgamma,\infty} = \decay_{\bsgamma,1}$.
\end{theorem}

In the proof we use the multi-index notation, which we recall here:
For $\bsnu=(\nu_j)^d_{j=1} \in \NN_0^d$ we write $|\bsnu| := \nu_1 + \cdots
+\nu_d$ and $\bsnu !:= \prod^d_{j=1} \nu_j !$.

\begin{proof}
 Obviously, we always have
\begin{equation*}
\gamma_d^{1/p} = \Gamma_1 \gamma_d^{1/p}
\le \sum_{d\in u\subseteq [d]} \gamma_u^{1/p}
\end{equation*}
and $\decay_{\bsgamma,\infty} \le \decay_{\bsgamma,1}$.

Now let us consider the case where $q=p^*$ and $T:= \sum^\infty_{j=1} \gamma_j^{1/p^*} <1$. Then
\begin{equation*}
 \sum_{d\in u\subseteq [d]} \gamma_u^{1/p^*} = \sum_{d\in u\subseteq [d]} \Gamma_{|u|}^{1/p^*}\prod_{j\in u} \gamma_j^{1/p^*}
\le C_{p^*}^{1/p^*} \sum_{d\in u\subseteq [d]} (|u|!)\prod_{j\in u} \gamma_j^{1/p^*}.
\end{equation*}
Similar as in \cite[Lemma~6.2]{KSS11} we now employ the multinomial formula
and the formula for (finite) geometric series to obtain
\begin{equation*}
\begin{split}
 \sum_{d\in u\subseteq [d]} \gamma_u^{1/p^*}
&\le C_{p^*}^{1/p^*} \sum_{\bsnu \in \N_0^d; \nu_d \neq 0} \frac{|\bsnu|!}{\bsnu !}
\prod_{j\in u} \gamma_j^{\nu_j/p^*}\\
&= C_{p^*}^{1/p^*} \sum_{\kappa =0}^\infty \left( \sum_{\bsnu \in \N_0^d; |\bsnu|
= \kappa} \frac{\kappa !}{\bsnu !} \prod_{j=1}^d \gamma_j^{\nu_j/p^*}
- \sum_{\bsnu \in \N_0^{d-1}; |\bsnu|
= \kappa} \frac{\kappa !}{\bsnu !} \prod_{j=1}^{d-1} \gamma_j^{\nu_j/p^*}
\right) \\
&= C_{p^*}^{1/p^*} \sum_{\kappa =0}^\infty \left[ \left( \sum^d_{j=1} \gamma_j^{1/p^*}
\right)^{\kappa} - \left( \sum^{d-1}_{j=1} \gamma_j^{1/p^*} \right)^{\kappa}
\right] \\
&\le C_{p^*}^{1/p^*} (1-T)^{-2} \gamma_d^{1/p^*}.
\end{split}
\end{equation*}
In particular, we showed that $\sum_{u\in \U} \gamma_u^{1/p*} < \infty$,
which implies that $\decay_{\bsgamma,\infty} \ge p^*$.

Let now $\Gamma_{k} \le C_q (k!)^{q}$ for some $q<p^*$, and let $p\in (q,p^*]$
with $\sum^\infty_{j=1} \gamma_j^{1/p}<\infty$. (Recall that this sum is
always finite if $p < p^*$.)
Let $S :=  \left( 2 \sum_{j=1}^\infty \gamma_j^{1/p} \right)^p$ and set $\gamma^\ast_j := \gamma_j/S$.
Then $\sum_{j=1}^\infty (\gamma_j^\ast)^{1/p} = 1/2 < 1$.
Set $\Gamma^\ast_{k} = S^{k} \Gamma_{k}$ for all $k\in\N_0$. Then there is a constant $C^\ast > 0$ such that $\Gamma^\ast_{k} = S^{k} \Gamma_{k} \le S^{k} C_q (k!)^{q}  \le C^\ast (k!)^p$. Thus, by the argument used in the case
$p^*=q$, we get
\begin{equation*}
 \sum_{d\in u\subseteq [d]} \gamma_u^{1/p} =
\sum_{d\in u\subseteq [d]} \left( \Gamma^*_{|u|} \prod_{j\in u} \gamma_j^* \right)^{1/p} = O(\gamma_d^{1/p}).
\end{equation*}
In particular, we showed that $\sum_{u\in \U} \gamma_u^{1/p} < \infty$
for all $p<p^*$, which implies that $\decay_{\bsgamma,\infty} \ge p^*$.
The same holds for $p=p^*$ if $\sum_{j=1}^\infty \gamma_j^{1/p} < \infty$.
\end{proof}

\begin{corollary}\label{Corollary}
Let $\bsgamma = (\gamma_u)_{u\in \U}$ be POD weights with $\gamma_u
= \Gamma_{|u|}\prod_{j\in u} \gamma_j$. Let $p^*:= \decay_{\bsgamma,1} >1$
and $q < p^*$. Let there exist a constant $C_q>0$ such that
$\Gamma_k \le C_q (k!)^q$ for all $k\in\N$.
Then we have for every $\tau\in [1,p^*)$ and every constant $\widetilde{C}_\tau>0$ that
\begin{equation*}
 \sum_{d\in u\subseteq [d]} \gamma_u^{1/\tau} \widetilde{C}_{\tau}^{|u|}
= \Theta(\gamma_d^{1/\tau}).
\end{equation*}
\end{corollary}

\begin{proof}
Let $\tau$ and $\widetilde{C}_\tau$ be given.
Obviously, $\sum_{d\in u\subseteq [d]} \gamma_u^{1/\tau} \widetilde{C}_{\tau}^{|u|}
= \Omega(\gamma_d^{1/\tau})$. Now let
$p\in (\max\{\tau,q\}, p^*)$. Define the POD weights $\widetilde{\bsgamma}
= \Gamma_{|u|}\prod_{j\in u} \widetilde{\gamma_j}$ by
$\widetilde{\gamma}_j = \gamma_j \widetilde{C}^\tau_\tau$. Then $p^* = \decay_{\widetilde{\bsgamma},1}$ and, due to Jensen's inequality and
Theorem \ref{Corollary2}, we obtain
\begin{equation*}
 \sum_{d\in u\subseteq [d]} \gamma_u^{1/\tau} \widetilde{C}^{|u|}_\tau
= \sum_{d\in u\subseteq [d]} \widetilde{\gamma}_u^{1/\tau}
\le \left( \sum_{d\in u\subseteq [d]} \widetilde{\gamma}_u^{1/p} \right)^{p/\tau}
= \Theta((\widetilde{\gamma}_d^{1/p})^{p/\tau}) = \Theta(\gamma_d^{1/\tau}).
\end{equation*}
\end{proof}

From Corollary \ref{Corollary} we immediately get the following useful corollary.

\begin{corollary}\label{Lemma3.15}
Let $\bsgamma$ be POD weights that satisfy the assumptions of
Corollary \ref{Corollary}, and let $v_k = v_k^{(2)} = [L_k]$
for all $k\in\N$. Let $\tau\in [1/2, \decay_{\bsgamma,1}/2)$.
Then we have for $C_{k,\tau,\gamma}$ as in (\ref{c_k_tau_gamma})
\begin{equation*}
C_{k,\tau,\gamma} = \Theta(\sigma_k),
\hspace{3ex}\text{where $\sigma_k:= \sum_{j=L_{k-1}+1}^{L_k} \gamma_j$,}
\end{equation*}
and furthermore
\begin{equation*}
 \sum_{j\notin U(m)} \widehat{\gamma}_{u_j}
= \Theta \left( \sum_{j=L_m+1}^\infty \gamma_j \right).
\end{equation*}
\end{corollary}

\subsubsection{Nested subspace sampling}

Let $\bsgamma$ be POD weights that satisfy the assumptions of
Corollary \ref{Corollary}.
%and have $\decay_{\bsgamma,1} >1$
Let $L_k: = L\lceil a^{k-1} \rceil$ for $k\in\N$, where
$L\in\N$ and $a\in(1,\infty)$ are fixed.
(A canonical
choice would be $L=1$ and $a=2$, but in some applications other
choices may be more convenient.)
Furthermore, let $v_k = v_k^{(2)} = [L_k]$
for all $k\in\N$. Let $\alpha \ge 1/2$.
We use multilevel algorithms $Q^{\ML}_m$ as in (\ref{multilevel-algo})
that employ quadratures
$Q_{v_k}$ fulfilling the estimate (\ref{assumption3.9}).
In particular, these multilevel algorithms satisfy
%for arbitrary $\tau\in [1/2,\min\{\alpha,\decay_{\bsgamma,1}/2\})$
the error estimate (\ref{errorestimate}).

\begin{theorem}\label{Theorem6}
Let $\$(k)=O(k^s)$ for some $s\ge 0$. Let $\bsgamma = (\gamma_u)_{u\in \U}$ be POD weights that satisfy the assumptions of
Corollary \ref{Corollary}.
%and have $\decay_{\bsgamma,1} >1$
We assume that there exists an $\alpha \ge 1/2$ such that for all $k\in\N$ and
all $n_k\in \N$ we find quadratures $Q_{v_k}$ as in (\ref{algobaustein})
that satisfy (\ref{assumption3.9}).  Then our multilevel algorithms $Q^{\ML}_m$,
defined as in (\ref{multilevel-algo}), establish the following result:

In the case where $s\ge (2\alpha-1)/2\alpha$ we obtain
\begin{equation}
\label{uppboupod}
p^{\nes} \le \max \left\{ \frac{1}{\alpha}, \frac{2s}{\decay_{\bsgamma,1} - 1} \right\}.
\end{equation}

In the case where $0 \le s < (2\alpha-1)/2\alpha$, we obtain for
\begin{itemize}
\item[] $\decay_{\bsgamma,1} \ge 2\alpha$:
\begin{equation*}
p^{\nes} \le \frac{1}{\alpha},
\end{equation*}
\item[] $2\alpha > \decay_{\bsgamma,1} > 1/(1-s)$:
\begin{equation*}
p^{\nes} \le \frac{2}{\decay_{\bsgamma,1}},
\end{equation*}
\item[] $1/(1-s) \ge \decay_{\bsgamma,1} >1$:
\begin{equation*}
p^{\nes} \le \frac{2s}{\decay_{\bsgamma,1}-1}.
\end{equation*}
\end{itemize}
\end{theorem}

If the assumptions of Theorem \ref{Theorem6} hold and if additionally
the $n$th minimal worst case error of univariate integration satisfies
$e(n; H(K)) = \Omega(n^{-\alpha})$, then, due to the lower bound on $p^{\nes}$ in
(\ref{neslowboupod}), we have  a sharp upper bound on the exponent $p^{\nes}$ if $s\ge (2\alpha-1)/2\alpha$, and for
$\decay_{\bsgamma,1} \ge 2\alpha$ and for $1/(1-s) \ge \decay_{\bsgamma,1} >1$
if  $0 \le s < (2\alpha-1)/2\alpha$.
Observe that the case $s\ge(2\alpha -1)/2\alpha$ is more interesting and relevant
than the case $0 \le s < (2\alpha-1)/2\alpha$, see, e.g., \cite{Gil08a, NHMR11, PW11}.

Notice further that Theorem \ref{Theorem6} improves on the corresponding
results in \cite{Gne10, NHMR11} for product weights. (Compare, e.g., Theorem \ref{Theorem6}
with \cite[Thm.~4.2]{Gne10} and \cite[Cor.~2]{NHMR11}, where the Wiener kernel
$K(x,y) = \min\{x,y\}$ is treated.)

\begin{proof}
Let $p \in (1,\decay_{\gamma,1})$ and let $\tau\in [1/2,\min\{\alpha, p/2\})$
satisfy (\ref{assumption3.9}).
(Here we treat in detail only the case $\alpha > 1/2$; in the easier case $\alpha =1/2$
one chooses always $\tau = 1/2$.)
Let $\sigma_k$ be as in Corollary \ref{Lemma3.15}. Then we get from (\ref{errorestimate})
and Corollary \ref{Lemma3.15} that
\begin{equation*}
 [e(Q^{\ML}_m;\Hg)]^2 = O \left( \sum^m_{k=1} \sigma_{k} n_k^{-2\tau}
+ \sum_{j= L_m +1}^\infty \gamma_{j} \right).
\end{equation*}
Let $m$ be given, and put $M:= \sum^m_{k=1}L_k^s$. For given cost $S \ge M$
of order $S= \Theta(L^s_m)$ we choose the number of sample points $n_k$ as $n_k := \lceil x_k \rceil$,
where
\begin{equation*}
 x_k = C \sigma_k^{\frac{1}{2\tau+1}} L_k^{-\frac{s}{2\tau+1}},
\hspace{3ex}\text{with}\hspace{3ex}
C= S \left( \sum^m_{k=1} \sigma_k^{\frac{1}{2\tau+1}} L_k^{\frac{2\tau s}{2\tau+1}}
\right)^{-1}.
\end{equation*}
The cost of the multilevel algorithm $Q^{\ML}_m$ is then of order
$\cost_{\nes}(Q^{\ML}_m) = O(S)$.
We get
\begin{equation*}
 \sum_{k=1}^m \sigma_k n_k^{-2\tau}
\le S^{-2\tau} \left( \sum_{k=1}^m \sigma_k^{\frac{1}{2\tau +1}}
L_k^{\frac{2\tau s}{2\tau +1}} \right)^{2\tau +1}.
\end{equation*}
Since $\sigma_k = O(L_{k-1}^{1-p})$ and $\sum_{j= L_m+1}^\infty \gamma_{j}
= O(L_m^{1-p})$, we obtain the error estimate
\begin{equation}\label{ausgangsgleichung}
 [e(Q^{\ML}_m;\Hg)]^2 = O \left( S^{-2\tau} \big(1 + L_m^{1- p +2s\tau} \big)
+ L_m^{1-p} \right) = O \left( S^{-2\tau} + S^{-\frac{p-1}{s}} \right).
\end{equation}

\emph{Case 1}: $s\ge (2\alpha -1)/2\alpha$.
Here we have two subcases.

\emph{Subcase 1a}: $p\ge 1+2\alpha s$.
This implies $(p-1)/s \ge 2\alpha$ and $p \ge 2\alpha$. Hence we obtain
\begin{equation}\label{Fall1a}
 [e(Q^{\ML}_m;\Hg)]^2 = O \left( S^{- 2\tau } \right),
\end{equation}
and we may choose $\tau$ arbitrarily close to $\alpha$.

\emph{Subcase 1b}: $1+ 2\alpha s> p >1$. Then it is not hard to verify that
$(p -1)/s \in (0,\min\{2\alpha, p\})$.
Thus we may choose $\tau \ge (p -1)/2s$ and get
\begin{equation}\label{Fall1b}
 [e(Q^{\ML}_m;\Hg)]^2 = O \left( S^{- \frac{p -1}{s} } \right).
\end{equation}
If we let $p$ tend to $\decay_{\bsgamma,1}$, we see that the estimates (\ref{Fall1a}) and (\ref{Fall1b}) imply (\ref{uppboupod}).

\emph{Case 2}: $(2\alpha -1)/2\alpha > s \ge 0$.
Here we have three subcases.

\emph{Subcase 2a}: $p\ge 2\alpha$. Then $(p -1)/s > 2\alpha$
and we get (\ref{Fall1a}), where we again can choose $\tau$ arbitrarily close
to $\alpha$.

\emph{Subcase 2b}: $2\alpha > p > 1/(1-s)$. Then $(p -1)/s > p$. Hence we get (\ref{Fall1a}) and may choose $\tau$ arbitrarily close to $p/2$.

\emph{Subcase 2c}: $1/(1-s) \ge p >1$. Then $2\alpha > p \ge (p-1)/s$.
Choosing $\tau \ge (p-1)/2s$, we obtain (\ref{Fall1b}).

Letting again $p$ tend to $\decay_{\bsgamma}$, we have thus verified the theorem.
\end{proof}

\subsubsection{Unrestricted subspace sampling}

If the cost function satisfies $\$(k) = O(k^s)$ for $0\le s \le 1$,
we may again use multilevel algorithms as done in the previous
subsection.
%or ``changing dimension algorithms'' as proposed for product weights in
%\cite{KSWW10, PW11}. We will describe the changing dimension algorithms
%from \cite{PW11} in detail in the
%higher order convergence setting in Section \ref{USSHIGHERORDER}.
In the case where we have a cost function $\$(k) = \Omega(k^s)$ for
$s \ge 1$ and product weights, changing dimension algorithms, as
considered in \cite{KSWW10, PW11}, have proved to be the essentially
optimal choice in the unrestricted subspace sampling setting,
see the analysis in \cite{PW11}.
We present here a slight modification of the changing dimension
algorithms from \cite{PW11} which ensures that the results from
\cite{PW11} do not only hold for product weights but for all
POD weights that satisfy the conditions of Corollary \ref{Corollary}.

As in \cite{PW11}, we assume that there exist positive constants
$c$, $C$, $\tau$, a non-negative $\lambda_1$, and a $\lambda_2\in [0,1]$ such
that for each $u\in\U\setminus \{\emptyset\}$ and $n\in\N$ there are algorithms $Q_{n,u}$
using $n$ function evaluations of functions $f_u\in H_u$ with
\begin{equation}
\label{CDA-Bausteine}
 e(Q_{n,u};H_u)^2 \le \frac{cC^{|u|}}{(n+1)^{2\tau}}
\left( 1 + \frac{\ln(n+1)}{(|u|-1)^{\lambda_2}} \right)^{\lambda_1(|u|-1)^{\lambda_2}},
\end{equation}
where by convention the last factor in (\ref{CDA-Bausteine}) should be $1$ for $|u|=1$.
We may assume that $c\ge 1$ and $C\ge C_0$, so that (\ref{CDA-Bausteine}) holds also
true for $n=0$.
With the help of the building blocks $Q_{n,u}$ one can define changing
dimension algorithms for a fixed $\lambda_0 \in (0, 1-1/\decay_\gamma)$ and
any given $\varepsilon >0$ in the
following way:
Let us put
\begin{equation}
\label{Ls}
%L(r;\bsgamma) := \sum_{j=1}^\infty \gamma_j^r
%\hspace{2ex}\text{and}\hspace{2ex}
L_r := \sum_{\emptyset \neq u \in \U} \gamma_u^r
\end{equation}
for suitable $r\ge 0$.
Choose $\tau$ such that
$\tau < \lambda_0\cdot \decay_{\bsgamma}/2$.
For each $u\in \U$ satisfying
$\gamma_u^{\lambda_0} L_{1-{\lambda_0}} c\, C^{|u|}  \le \varepsilon^2$
we choose
$n_u = n_u(\e, \lambda_0)$ to be zero and $Q_{n_u,u}$ to be the trivial zero
algorithm $Q_{n_u,u}f_u = 0$ for all $f_u\in H_u$.
Otherwise, we put
$n_u = \lfloor (\gamma^{\lambda_0}_u L_{1-{\lambda_0}}
c\,C^{|u|} \e^{-2})^{1/2\tau} \rfloor$
and choose $Q_{n_u,u}$ as in (\ref{CDA-Bausteine}).
We define the changing dimension algorithm $Q^{\CD}_{\e}$ by
\begin{equation}
\label{def-CD}
Q^{\CD}_{\e}(f) = f(\bsc) + \sum_{\emptyset \neq u \in \U}
Q_{n_u,u}(f_u).
\end{equation}
Observe that for any $\varepsilon >0$ there are only finitely many $u\in \U$ with
$n_u \ge 1$.
For given $\varepsilon>0$ let
\begin{equation*}
d(\e) := \max\left\{ \ell \in \N \,|\, c\,C^{\ell}
L_{1-{\lambda_0}}\gamma^{\lambda_0}_{[\ell]} > \varepsilon^2 \right\}.
\end{equation*}
Then it is easily verified that $|u|> d(\varepsilon)$ implies $n_u=0$.
Thus the ``$\e$-dimension'' $d(\e)$ is the largest number of active variables used by the
changing dimension algorithm $Q^{\CD}_{\e}$.
Due to Section \ref{CDA} we obtain
\begin{equation*}
\cost_{\unr}(Q^{\CD}_{\e})
\le \$(0) + \sum_{\emptyset \neq u \in \U} 2^{|u|}\$(|u|) n_u
\le  \$(0) + \$(d(\e)) \sum_{\ell =1}^{d(\e)} 2^{\ell}\sum_{|u|=\ell} n_u.
\end{equation*}
The following theorem is a slight generalization of \cite[Thm.~1]{PW11}.

\begin{theorem}\label{Thm_hoc_cda}
Let $\bsgamma = (\gamma_u)_{u\in \U}$ be POD weights that satisfy the assumptions of
Corollary \ref{Corollary}. Let $\lambda_0 \in (0, 1-1/\decay_\gamma)$, and  let
$\tau <\lambda_0 \cdot \decay_{\bsgamma}/2$ satisfy (\ref{CDA-Bausteine}).
Then the changing dimension algorithm $Q^{\CD}_{\e}$ defined in
(\ref{def-CD}) satisfies
\begin{equation*}
e(Q^{\CD}_{\e};\Hg) \le \e^{1-o(1)}
\hspace{2ex}\text{as $\e\to 0$,}
\end{equation*}
and its cost satisfies
%for arbitrarily small $\delta >0$
\begin{equation*}
\cost_{\unr}(Q^{\CD}_{\e})
%\le \left( L_{1-\beta}\, \exp \left( 2 C^{1/2}_{b,\alpha,1/\tau}
%L \big( \beta/(2\tau);\bsgamma \big) \right) \right)\$(d(\e))
%\, \e^{-1/\tau},
= O\left( \$(d(\varepsilon))
\varepsilon^{-1/\tau} \right),
\end{equation*}
where
\begin{equation*}
d(\e)
= O \left( \frac{\ln(1/\e)}{\ln\ln(1/\e)} \right)
= o(\ln(1/\e)).
\end{equation*}
%where $L_{1-a_0}$ $L(\beta/(2\tau);\bsgamma)$ as defined in (\ref{Ls}).
If the cost function $\$$ satisfies
$\$(d) = O(e^{\ell d})$ for some $\ell \ge 0$, then
the integration problem is strongly tractable with exponent
\begin{equation*}
p^{\unr} \le \max \left\{ \frac{1}{\tau},
\frac{2}{\decay_{\bsgamma}-1} \right\}.
\end{equation*}
Let us now additionally assume that $\$(d) = \Omega(d)$ and that
the $n$th minimal worst case error of univariate integration satisfies
$e(n;H(K))= \Omega(n^{-\alpha})$. If (\ref{CDA-Bausteine}) holds for
$\tau$ arbitrarily close to $\alpha$, then
\begin{equation*}
p^{\unr} = \max \left\{ \frac{1}{\alpha},
\frac{2}{\decay_{\bsgamma}-1} \right\}.
\end{equation*}
\end{theorem}

%\begin{proof}
In the case of product weights, the statement of Theorem \ref{Thm_hoc_cda}
was proved in \cite{PW11}, see Theorem 1 and 2 there.

In the case where we have general POD weights satisfying the assumptions of
Corollary \ref{Corollary}, we see that $\decay_{\bsgamma,\infty} = \decay_{\bsgamma,1}$,
see Theorem \ref{Corollary2}, and these quantities do not change if we multiply
the $\gamma_j$, $j\in\N$, by some constant.
With the help of this observation one can verify that for the upper bound on
$p^{\unr}$ the analysis in \cite{PW11}
only needs to be slightly modified to carry over to POD weights that satisfy the
assumptions of Corollary \ref{Corollary}. The lower bound follows from
(\ref{neslowboupod}).
%\end{proof}

%To illustrate the qualitative difference between changing dimension
%and multilevel algorithms in the POD weight setting, let us assume that
%$\decay_{\bsgamma}$ is only moderate. Then, if on the one hand
%$\$(k) = \Omega(k^s)$ for $s>1$, changing dimension algorithms
%outperform multilevel algorithms in the unrestricted subspace sampling
%regime.
%If on the other hand $\$(k) =O(k^s)$ for $s<1$, then multilevel
%algorithms outperform changing dimensions in both sampling models.
%In the case where $\$(k) = \Theta(k)$, both types of algorithms
%achieve the same (nearly optimal) convergence rates in the unrestricted
%subspace sampling;
%cf. with the concrete results in Section \ref{HOC-POD}.

\subsection{Weights with finite algorithmic dimension}
\label{UB_FAD}

Let $\mathcal{W}\subseteq \U$ with
minimal algorithmic dimension $d\in\N$, and let $(\gamma_u)_{u\in \U}$
be weights with $\gamma_u = 0$ for all $u\notin \mathcal{W}$ (i.e.,
$\AC = \AC(\bsgamma) \subseteq \mathcal{W}$).
Assume furthermore, that there exist non-negative constants $c,C,\beta_1,\beta_2$,
an $\alpha >0$, and for any $n\in\N_0$ a
quadrature $Q_n$, given by
\begin{equation}
\label{algobaustein-FAD}
Q_n(f) = \sum^n_{i=1} a_i^{(n)} f(\bst^{(i,n)})
\hspace{2ex}
%\text{for all $f\in H_{\bsgamma, [d]}$,}
\text{with $a_i^{(n)}\in\R$, $\bst^{(i,n)} \in D^{d}$,}
\end{equation}
such that
\begin{equation}
\label{assuconv}
e((Q_n)_{u}; H_u) \le
cC^{|u|}\,(n+1)^{-\alpha}\, (1+\ln(n+1))^{\beta_1 |u| + \beta_2}
\hspace{2ex}\text{for all $u\subseteq [d]$.}
\end{equation}
With the help of the algorithms $Q_n$ and a mapping $\phi$ that
satisfies
(\ref{phi}), we can construct for arbitrary $v\in \U$
algorithms $Q^{\mathcal{W}}_{v}$ on
$\Hg$ in the following way (cf. also \cite[Prop.~3.11]{Gne10}):
First we formally consider infinite vectors
\begin{equation*}
\bst^{(i,n)}_{\infty} \in D^\N,
\hspace{2ex}\text{where the $j$th component is}\hspace{2ex}
t^{(i,n)}_{\infty,j} := t^{(i,n)}_{\phi(j)}.
\end{equation*}
Then we define  the quadrature
$Q^{\mathcal{W}}_{n,v}$ by
\begin{equation}
\label{algoW}
Q^{\mathcal{W}}_{n,v}(f) := \sum^n_{i=1} a_i^{(n)}
f(\bst^{(i,n)}_{\infty,v};\bsc)
\hspace{2ex}\text{for all $f\in\Hg$.}
\end{equation}
Note that for $u\subseteq v$, $u\in \mathcal{W}$, we have $|u|=|\phi(u)|$
and $e((Q^{\mathcal{W}}_{n,v})_u; H_u) = e((Q_n)_{\phi(u)};H_{\phi(u)} )$.
By combining such algorithms in a suitable way, we get the
following results for nested and unrestricted subspace
sampling.

\subsubsection{Nested subspace sampling}

In the nested and in the unrestricted subspace sampling regime
we propose to use multilevel algorithms $Q^{\ML}_m$ that employ
the quadratures $Q_{v_k} = Q^{\mathcal{W}}_{n,v_k}$ defined in (\ref{algoW}).
Here we consider for the $k$th level the set of coordinates
$v_k = v_k^{(1)} = \cup_{j\in [L_k]} u_j$ and
$L_k:= L \lceil a^{k-1} \rceil$,
where $L\in \N$ and $a\in (1,\infty)$ are fixed.
As in (\ref{algobaustein}),
the quadrature $\widehat{Q}_k$ on the $k$th level is given by
\begin{equation*}
\widehat{Q}_k(f) := Q^{\mathcal{W}}_{n_k,v_k}(f-\Psi_{v_{k-1}}f).
\end{equation*}
Due to (\ref{worMLid}) and (\ref{assuconv}) we get for arbitrarily
small $\delta >0$
\begin{equation*}
\begin{split}
[e(Q^{\ML}_m;\Hg)]^2 &=
\sum^m_{k=1}\sum_{j\in V_k} \gamma_{u_j}
[e((Q^{\mathcal{W}}_{n_k, v_k})_{u_j};H_{u_j})]^2
+ \sum_{j\notin U(m)} \widehat{\gamma}_{u_j}\\
&\le \widetilde{C}^2 \sum^m_{k=1} \left( \sum^{L_k}_{j=L_{k-1}+1}
\gamma_{u_j} \right) (n_k+1)^{2(\delta- \alpha)}
+\tail_{\bsgamma}(L_m),
\end{split}
\end{equation*}
where the constant $\widetilde{C}$ depends on
$d, \alpha,\delta, c,C, \beta_1$, and $\beta_2$, but not on $m$ or the
specific values $n_k$, $k=1,\ldots,m$. Notice that in the last
inequality we implicitly used $n_1\ge n_2 \ge \cdots \ge n_m$,
since it might happen for some $L_{k-1} < j \le L_k$ that
$u_j\subseteq v_l$ for an $l<k$.

This estimate is almost identical with estimate (45) in \cite[Sect.~3.2.2]{Gne10}:
there one just has to replace $n_k$ by $n_k+1$ and $\delta - 1$ by $\delta-\alpha$,
and rename the constant $C_{\eta,\omega,\delta}$ by $\widetilde{C}^2$. Adapting
the reasoning in \cite{Gne10} that follows after estimate (45), we obtain the
following theorem.

\begin{theorem}
\label{UppBouFAD}
Let $\$(k) = O(k^s)$ for some $s\ge 0$. Let the weights $\bsgamma$ have finite
algorithmic dimension, and let $\decay_{\bsgamma} >1$.
Assume that there exist for $\alpha >0$ and all $n\in\N$ algorithms $Q_n$ as in
(\ref{algobaustein-FAD}) that satisfy (\ref{assuconv}). For $k=1,2,\ldots$,
let $Q_{v_k} = Q^{\mathcal{W}}_{n,v_k}$ be as in (\ref{algoW}).
Then the multilevel algorithms $Q^{\ML}_m$, defined as in (\ref{multilevel-algo}),
establish the following result:
The exponent of strong tractability in the nested subspace sampling
model satisfies
\begin{equation}
\label{uppboufad}
p^{\nes} \le \max \left\{ \frac{1}{\alpha}, \frac{2s}{\decay_{\bsgamma} - 1} \right\}.
\end{equation}
\end{theorem}

If the assumptions of Theorem \ref{UppBouFAD} hold and if additionally
the $n$th minimal worst case error of univariate integration satisfies
$e(n; H(K)) = \Omega(n^{-\alpha})$, then, due to the lower bound on $p^{\nes}$ in
(\ref{neslowboupod}), we see that our upper bound on $p^{\nes}$ in (\ref{uppboufad})
is sharp for finite-intersection weights; cf. also
%\cite[Thm.~3.2]{Gne10}, where the same result was proved for the special case of the Wiener kernel $K(x,y)=\min\{x,y\}$, and
Section \ref{HOC-POD-NEST}.

\subsubsection{Unrestricted subspace sampling}
\label{SUBSEC_4.4.2}

In the case where the cost function $\$$ is of the form $\$(k) = \Omega(k^s)$ for some
$s>1$, we can improve the bound on the exponent of tractability from Theorem
\ref{UppBouFAD} by changing from the nested to the more generous unrestricted
subspace sampling model. For general finite-order weights $\bsgamma$ of order $\omega$ appropriate changing dimension algorithms were provided
in \cite{KSWW10}. These algorithms can in particular be used
for weights with finite algorithmic dimension $d$, which are
finite-order weights of order $\omega = d$.
If $\decay_{\bsgamma, \omega} > 1$ and if there exist algorithms $Q_n$ as in
(\ref{algobaustein-FAD}) satisfying (\ref{assuconv}), then
changing dimension algorithms lead to an upper bound
\begin{equation}
\label{ksww10}
p^{\unr} \le \max \left\{ \frac{1}{\alpha}, \frac{2}{\decay_{\bsgamma}-1} \right\},
\end{equation}
see \cite[Thm.~5(a) \& Sect.~5.7]{KSWW10}.
Together with Theorem \ref{UppBouFAD} this implies the following result.

\begin{theorem}
\label{UppBouFADunr}
Let $\$(k) = O(k^s)$ for some $s\ge 0$. Let the weights $\bsgamma$ have finite
algorithmic dimension, and let $\decay_{\bsgamma} >1$. Assume that there exists
for some $\alpha >0$ and all $n\in\N$ algorithms $Q_n$ as in (\ref{algobaustein-FAD})
that satisfy assumption (\ref{assuconv}).
%Assume that the $n$th minimal worst case error of integration for the space $H$ of %univariate functions is of order $n^{-\alpha}$.
Then the exponent of tractability in the unrestricted subspace sampling model
satisfies
\begin{equation}
\label{uppboufadunr}
p^{\unr} \le \max \left\{ \frac{1}{\alpha}, \frac{2\min\{1,s\}}{\decay_{\bsgamma}-1} \right\}.
\end{equation}
\end{theorem}

Our lower bound on $p^{\unr}$ in (\ref{neslowboufad}) shows that the upper bound
(\ref{uppboufadunr}) is sharp for the sub-class of finite-intersection weights
if $e(n;H(K)) = \Omega(n^{-\alpha})$.

For finite-intersection weights and the Wiener kernel
$K(x,y) = \min\{x,y\}$ the bound (\ref{uppboufadunr}) was proved in
\cite[Thm.~3.12]{Gne10}.

\section{Higher Order Convergence}
\label{HOC}

In this section we confine ourselves to the domain $D=[0,1]$, endowed with
the restricted Lebesgue measure. We assume that $\alpha \ge 1$ is an integer.

\subsection{Higher order polynomial lattice rules}

Here we introduce polynomial lattice rules which can achieve arbitrary high convergence rates of the integration error for suitably smooth functions, see \cite{DP07}.

Classical polynomial lattices were introduced in \cite{nie92} (see also \cite[Section 4.4]{niesiam}) by Niederreiter. These lattices are obtained from rational functions over finite fields. For a prime $b$ let $\FF_b((x^{-1}))$ be the field of formal Laurent series over $\FF_b$. Elements of $\FF_b((x^{-1}))$ are formal Laurent series, $$L=\sum_{l=w}^{\infty}t_l x^{-l},$$ where $w$ is an arbitrary integer and all $t_l \in \FF_b$.  Note that $\FF_b((x^{-1}))$ contains the field of rational functions over $\FF_b$ as a subfield. Further let $\FF_b[x]$ be the set of all polynomials over $\FF_b$.

The following definition is a slight generalization of the definition from \cite{nie92}, see also \cite{niesiam}, which first appeared in \cite{DP07}; see also \cite[Chapter~15.7]{DP12}.

\begin{definition}
Let $b$ be prime and $1 \le m \le n$. Let $\vartheta_n$ be the map from $\FF_b((x^{-1}))$ to the interval $[0,1)$ defined by $$\vartheta_n\left(\sum_{l=w}^{\infty}t_l x^{-l}\right)=\sum_{l=\max(1,w)}^n t_l b^{-l}.$$

For a given dimension $s \ge 1$, choose an irreducible polynomial $p \in \FF_b[x]$ with $\deg(p)=n \ge 1$ and let $\bsq  = (q_1,\ldots ,q_s) \in (\FF_b[x])^s$. For $0 \le h<b^m$ let $h=h_0+h_1 b +\cdots +h_{m-1}b^{m-1}$ be the $b$-adic expansion of $h$. With each such $h$ we associate the polynomial $$h(x)=\sum_{r=0}^{m-1}h_r x^r \in \FF_b[x].$$ Then $\cS_{p,m,n}(\bsq)$ is the point set consisting of the $b^m$ points $$\bsx_h=\left(\vartheta_n\left(\frac{h(x) q_1(x)}{p(x)}\right),\ldots ,\vartheta_n\left(\frac{h(x) q_s(x)}{p(x)}\right)\right) \in [0,1)^s,$$ for $0 \le h < b^m$. An equal quadrature rule $\frac{1}{N} \sum_{h=0}^{N-1} f(\bsx_h)$ using the point set $\cS_{p,m,n}(\bsq) = \{\bsx_0, \bsx_1,\ldots, \bsx_{b^m-1}\}$ is called a polynomial lattice rule.
\end{definition}

We call $\bsq$ the generating vector of the polynomial lattice rule and $p$ the modulus. For more information on (higher order) polynomial lattice rules see \cite{DP07,DP12}.

Let  $x = \sum_{i=1}^\infty \tfrac{x_i}{b^i} \in [0,1)$ and let $\sigma = \sum_{i=1}^\infty \tfrac{\sigma_i}{b^i} \in [0,1)$, where $x_i, \sigma_i \in \{0, \ldots, b-1 \}$. We define the digital $b$-adic shifted point $y$ by
\[
   y = x \oplus \sigma = \sum_{i=1}^\infty \tfrac{y_i}{b^i},
\]
where $y_i = x_i + \sigma_i \in \integer_b$. For points $\bsx \in [0,1)^s$ and $\bssigma \in [0,1)^s$ the digital $b$-adic shift $\bsx \oplus \bssigma$ is defined component wise.

\begin{definition}\label{defshiftednet}
A  polynomial lattice rule $Q_{\bsq,p}$ for which the underlying quadrature points are digitally shifted by the same  $\bssigma \in [0,1)^s$ is called a digitally shifted polynomial lattice rule or simply a shifted polynomial lattice rule $Q_{\bsq, p}(\bssigma)$.
\end{definition}

\subsection{Reproducing kernel of smoothness $\alpha$}

Let $c \in [0,1]$ and let $\alpha \ge 1$ be an integer. We consider the anchored reproducing kernel for smooth functions anchored at $c$ given by (see
\cite[Example~4.2]{KSWW10a})
\begin{equation*}
K_{\alpha,c}(x,y) =  \left\{\begin{array}{ll} \sum_{r=1}^{\alpha-1} \frac{(x-c)^{r}}{r!}
\frac{(y-c)^r}{r!} +  \int_c^1
\frac{(x-t)_+^{\alpha-1}}{(\alpha-1)!}
\frac{(y-t)_+^{\alpha-1}}{(\alpha-1)!} \,\mathrm{d} t,  & \mbox{if
} x, y > c, \\ \sum_{r=1}^{\alpha-1} \frac{(x-c)^{r}}{r!}
\frac{(y-c)^r}{r!} + \int_0^c \frac{(t-x)_+^{\alpha-1}}{(\alpha-1)!}
\frac{(t-y)_+^{\alpha-1}}{(\alpha-1)!} \,\mathrm{d} t, & \mbox{if }
x, y < c, \\ 0 & \mbox{otherwise},
\end{array} \right.
\end{equation*}
where $(x-t)_+ = \max(x-t,0)$ and $(x - t)_+^0 := 1_{x  > t}$ and for $\alpha=1$ the empty sum $\sum_{r=1}^{\alpha-1}$ is defined as $0$. The
inner product of the corresponding reproducing kernel Hilbert space
$H(K_{\alpha,c})$ is given by
\begin{equation*}
\langle f, g \rangle_{H(K_{\alpha,c})} = \sum_{r=1}^{\alpha-1} f^{(r)}(c)
g^{(r)}(c) + \int_0^1 f^{(\alpha)}(x) g^{(\alpha)}(x) \,\mathrm{d}
x,
\end{equation*}
with corresponding norm $\|\cdot \|_{H(K_{\alpha,c})} = \sqrt{ \langle \cdot, \cdot \rangle_{H(K_{\alpha, c})}}$. Note that for every $f \in H(K_{\alpha, c})$ we have $f(c) =0$.

It is well known that the $n$th minimal error of univariate integration on $H(K_{\alpha,c})$
is of order
\begin{equation}\label{univariate_convergence}
 e(n;H(K_{\alpha,c})) = \Omega(n^{-\alpha}).
\end{equation}
%See \cite{T03} and the reference therein for analogous results. For instance, this result can be shown in the following way: For given points $t_1, \ldots, t_n \in [0,1]$ choose the integer $k$ such that $2^{k-2} \le n < 2^{k-1}$. Define $$g(x) = \left\{\begin{array}{rl} x^{\alpha+1} (1-x)^{\alpha+1} & \mbox{if } 0 \le x < 1, \\ 0 & \mbox{otherwise}. \end{array} \right.$$ The function $g$ is $\alpha$ times continuously differentiable. Then one proves that the integral of the function $f(x) = \sum_{z} g(x2^k-z)$, where the sum is over all integers $0 \le z < 2^k$ such that $[z 2^{-k}, (z+1) 2^{-k})$ does not contain $c, t_1,\ldots, t_n$, is bounded below by a constant, and that its norm is bounded above by a constant times $2^{\alpha k}$. Since $\sum_{i=1}^n a_i f(t_i) = 0$ by the definition of $f$,  the result then follows. %{\bf [REFERENZ?????]}.
%TODO

% It can be checked that for all $x,y \in [0,1]$ we have
% \begin{equation}\label{eq_eta}
% \eta_{\alpha,c}(x,y) = \int_c^x \int_c^y \eta_{\alpha-1,c}(u,v)
% \,\mathrm{d} u \mathrm{d} v.
% \end{equation}

\subsection{Embedding theorem}
\label{EMBEDDING}

We now investigate the decay of the Walsh coefficients for functions in $H(K_{\alpha,c})$. To do so, we briefly introduce Walsh functions in base $b$ \cite{chrest, Fine, walsh}. Let $b \ge 2$ be an integer and let $\omega_b = \mathrm{e}^{2\pi \mathrm{i} / b}$ be the $b$-th root of unity. For a nonnegative integer $k$ let $k = \kappa_0 + \kappa_1 b + \cdots + \kappa_{a-1} b^{a-1}$ denote the $b$-adic representation of $k$ and for $x \in [0,1)$ let $x = \xi_1 b^{-1} + \xi_2 b^{-2} + \cdots$ denote the $b$-adic representation of $x$, where we assume that infinitely many $\xi_i$ are different from $b-1$. Then the $k$th Walsh function in base $b$ is given by
\begin{equation*}
\wal_k(x) = \omega_b^{\kappa_0 \xi_1 + \kappa_1 \xi_2 + \cdots + \kappa_{a-1} \xi_a}.
\end{equation*}
For a function $f$ defined on $[0,1]$ we define the $k$th Walsh coefficient by
\begin{equation*}
\widehat{f}(k) = \int_0^1 f(x) \overline{\wal_k(x)} \,\mathrm{d} x.
\end{equation*}
See also \cite[Chapter~14, Appendix~A]{DP12} for more information on Walsh functions in the context of numerical integration.

Let $k = \kappa_1 b^{a_1-1} + \cdots + \kappa_\nu b^{a_\nu-1}$ with
$a_1 > \cdots > a_\nu > 0$ and $\kappa_1,\ldots, \kappa_\nu \in
\{1,\ldots, b-1\}$. Set
\begin{equation*}
\mu_\alpha(k) = \left\{\begin{array}{ll} 0 & \mbox{if } k = 0, \\
a_1 + \cdots + a_{\min(\alpha,\nu)} & \mbox{if } k > 0.
\end{array} \right.
\end{equation*}

For $\alpha \ge 2$ let $\mathcal{W}_\alpha$ denote the space of all Walsh series
$f:[0,1) \to\mathbb{R}$ given by
\begin{equation*}
f(x) = \sum_{k=1}^\infty \widehat{f}(k) \wal_k(x),
\end{equation*}
with
\begin{equation*}
\|f\|_{\mathcal{W}_\alpha} := \sup_{k \in \mathbb{N}}
|\widehat{f}(k)|  b^{\mu_\alpha(k)} < \infty.
\end{equation*}

It was shown in \cite[Lemma~3]{D09} that there is a constant $C_{1,r}
> 0$ such that
\begin{equation*}
\left| \int_0^1 \frac{x^r}{r!} \overline{\wal_k(x)} \,\mathrm{d} x
\right| \le \left\{\begin{array}{ll} 0 & \mbox{if } \nu > r, \\ C_{1,r}
b^{-\mu_r(k)} & \mbox{if } 0 \le \nu \le r
\end{array} \right\} \le C_{1,r} b^{-\mu_{\alpha}(k)}.
\end{equation*}
The constant $C_{1,r}$ can be chosen as
\begin{equation}\label{const_c1}
C_{1,r} = r! \left(\frac{3}{2\sin \pi/b}\right)^r \left(1 + \frac{1}{b} + \frac{1}{b(b+1)}\right)^{r-1}.
\end{equation}
Thus, there is a constant $C_{2,\alpha} > 0$ such that
\begin{equation*}
\left| \sum_{r=1}^{\alpha-1} \int_0^1 \frac{(x-c)^r}{r!}
\overline{\wal_k(x)} \,\mathrm{d} x \int_0^1 \frac{(y-c)^r}{r!}
\wal_l(y) \,\mathrm{d} y \right| \le C_{2,\alpha} b^{-\mu_{\alpha}(k) -
\mu_{\alpha}(l)}.
\end{equation*}
We can choose $$C_{2,\alpha} = \sum_{r=1}^{\alpha-1} C_{1,r}^2.$$

For $k \in \mathbb{N}_0$ let $J_k(x) = \int_0^x
\overline{\wal_k(t)}\,\mathrm{d} t$. Note that for $k > 0$ we have
$J_k(0) = J_k(1) = 0$. The following result goes back to
Fine~\cite{Fine} (see also \cite[Lemma~1]{D09}). The function
$J_k(x)$ can be represented by a Walsh series
\begin{equation*}
J_k(x) = \sum_{m=0}^\infty r_k(m) \wal_k(x),
\end{equation*}
where for $k \in \mathbb{N}$ with $k =  \kappa_1 b^{a_1-1} + \cdots + \kappa_\nu b^{a_\nu-1}$ and $k' = k-\kappa_1 b^{a_1-1}$ we have
\begin{equation*}
r_{k}(m) = \left\{\begin{array}{ll} b^{-\mu_1(k)}
(1-\omega_b^{-\kappa_1})^{-1} & \mbox{if } m = k', \\
b^{-\mu_1(k)} (1/2+(\omega_b^{-\kappa_1}-1)^{-1}) & \mbox{if } m =
k, \\ b^{-\mu_1(m)} (\omega_b^{\theta}-1)^{-1} & \mbox{if } m =
\theta b^{a_1+a+1} + k, \\ 0 & \mbox{otherwise}.
\end{array} \right.
\end{equation*}
For $k=0$ we have
\begin{equation*}
r_{0}(m) = \left\{\begin{array}{ll}  b^{-\mu_1(m)}
(\omega_b^{\theta}-1)^{-1} & \mbox{if } m = \theta b^{a+1},
\\ 0 & \mbox{otherwise}.
\end{array} \right.
\end{equation*}
%The function $r_k(m)$ is nonzero if $m$ is either $k'$, $k$, or $\theta b^{a+a_1+1} + k$. %For all other cases there is a constant $C > 0$ such that $|r_k(m)| \le C b^{-\mu_1(m)}$.

For $k \in \mathbb{N}_0$ and $\alpha =1$ let $\chi_1(k) = \int_0^1
1_{[t,1]}(x) \overline{\wal_k(x)} \,\mathrm{d} x$ and for $\alpha >
1$ let
\begin{equation*}
\chi^{(+)}_\alpha(k) = \int_0^1 (x-t)_+^{\alpha-1}
\overline{\wal_k(x)}\,\mathrm{d} x
\end{equation*}
and
\begin{equation*}
\chi^{(-)}_\alpha(k) = \int_0^1 (t-x)_+^{\alpha-1}
\overline{\wal_k(x)}\,\mathrm{d} x.
\end{equation*}

\begin{lemma}\label{lem_bound_taylor}
For $\alpha \in \mathbb{N}$ and $t \in [0,1]$ we have
\begin{equation*}
\left|\chi^{(+)}_\alpha(k) \right|, \left|\chi^{(-)}_\alpha(k)
\right| \le C b^{-\mu_{\alpha}(k)} \quad \mbox{for all } k \in
\mathbb{N}_0.
\end{equation*}
\end{lemma}

\begin{proof}
We show the result by induction. Let $\alpha = 1$. Then $(x-t)_+^0 = 1_{[t,1]}(x)$ and $(t-x)_+^0 = 1_{[0,t]}(x)$ and therefore the result follows from
\cite[Lemma~1]{D09}. Assume now the result holds for some $\alpha
\in \mathbb{N}$. Let $k \in \mathbb{N}$,  $k = \kappa_1 b^{a_1-1} + \cdots + \kappa_\nu b^{a_\nu-1}$ and $k' = k- \kappa_1 b^{a_1-1}$ with $\kappa_1, \ldots, \kappa_\nu \in \{1,\ldots, b-1\}$, $a_1 > a_2 > \cdots > a_\nu >0$ and $0 \le k' <
b^{a_1-1}$. Then
\begin{eqnarray*}
\chi^{(+)}_{\alpha+1}(k) & = & \int_0^1 (x-t)_+^{\alpha}
\overline{\wal_k(x)} \,\mathrm{d} x \\ & = & J_k(x) (x-t)_+^{\alpha}
\mid_{x=0}^1 - \alpha \int_0^1 (x-t)_+^{\alpha-1} J_k(x)
\,\mathrm{d} x \\ & = & - \alpha \int_0^1 (x-t)_+^{\alpha-1} J_k(x)
\,\mathrm{d} x \\ & = & -\alpha \sum_{m=0}^\infty r_k(m) \int_0^1
(x-t)_+^{\alpha-1} \overline{\wal_m(x)} \,\mathrm{d} x \\ & = &
-\alpha \sum_{m=0}^\infty r_k(m) \chi^{(+)}_\alpha(m).
\end{eqnarray*}
Thus there is some constant $C > 0$ such that
\begin{equation*}
|\chi^{(+)}_{\alpha+1}(k)| \le C \alpha \left(b^{-\mu_1(k) -
\mu_{\alpha}(k')} +  b^{-\mu_1(k) - \mu_{\alpha}(k)} + b^{-\mu_1(k)
- \mu_{\alpha}(k)} \sum_{a=1}^\infty b^{-a} \right) \le C'_\alpha
b^{-\mu_{\alpha+1}(k)}.
\end{equation*}
The result for $\chi^{(-)}_{\alpha+1}$ can be shown by the same
arguments.
\end{proof}

By keeping track of the constant in Lemma~\ref{lem_bound_taylor} one can show that the constant can be chosen as $C_{1,\alpha}$ given by \eqref{const_c1}.

We now prove the following continuous embedding.
\begin{theorem}\label{thm_embedding}
Let $\alpha \in \mathbb{N}$ with $\alpha \ge 2$. There is a constant $C > 0$ such that for all $f \in
H(K_{\alpha,c})$ we have
\begin{equation*}
\|f\|_{\mathcal{W}_\alpha} \le C \|f\|_{H(K_{\alpha,c})}.
\end{equation*}
Thus we have the continuous embedding
\begin{equation*}
H(K_{\alpha,c}) \hookrightarrow \mathcal{W}_\alpha.
\end{equation*}
\end{theorem}

\begin{proof}
Let $f \in H(K_{\alpha,c})$. Then for $x \in [c,1]$ we
have the Taylor series expansion with integral remainder
$$f(x) = \langle f, K_{\alpha,c}(\cdot, x) \rangle_{H(K_{\alpha,c})} =
\sum_{r=1}^{\alpha-1} f^{(r)}(c) (x-c)^r + \int_c^1 f^{(\alpha)}(t)
\frac{(x-t)_+^{\alpha-1}}{(\alpha-1)!} \,\mathrm{d} t$$ and for $x
\in [0,c]$ we have the Taylor series expansion with integral remainder
$$f(x) = \langle f, K_{\alpha,c}(\cdot, x) \rangle_{H(K_{\alpha,c})} =
\sum_{r=1}^{\alpha-1} f^{(r)}(c) (x-c)^r + \int_0^c f^{(\alpha)}(t)
\frac{(t-x)_+^{\alpha-1}}{(\alpha-1)!} \,\mathrm{d} t.$$ Therefore
\begin{eqnarray*}
\widehat{f}(k) & = & \int_0^1 f(x) \overline{\wal_k(x)} \,\mathrm{d}
x \\ & = & \sum_{r=1}^{\alpha-1} f^{(r)}(c) \int_0^1 (x-c)^r
\overline{\wal_k(x)}\,\mathrm{d} x + \int_0^1 \int_0^1 1_{[c,1]}(t)
f^{(\alpha)}(t) \frac{(x-t)_+^{\alpha-1}}{(\alpha-1)!}
\overline{\wal_k(x)} \,\mathrm{d} t \,\mathrm{d} x \\ && + \int_0^1
\int_0^1 1_{[0,c]}(t) f^{(\alpha)}(t)
\frac{(t-x)_+^{\alpha-1}}{(\alpha-1)!} \overline{\wal_k(x)}
\,\mathrm{d} t \,\mathrm{d} x \\ & = &  \sum_{r=1}^{\alpha-1}
f^{(r)}(c) \int_0^1 (x-c)^r \overline{\wal_k(x)}\,\mathrm{d} x +
\int_0^1  1_{[c,1]}(t) f^{(\alpha)}(t) \chi_\alpha^{(+)}(k) \,\mathrm{d} t \\
&& + \int_0^1  1_{[0,c]}(t) f^{(\alpha)}(t) \chi_\alpha^{(-)}(k)
\,\mathrm{d} t.
\end{eqnarray*}
Thus, using \cite[Lemma~3]{D09} and Lemma~\ref{lem_bound_taylor}
there is some constant $C > 0$ such that
\begin{eqnarray*}
|\widehat{f}(k)| & \le & \sum_{r=1}^{\alpha-1} |f^{(r)}(c)|
\left|\int_0^1 (x-c)^r \overline{\wal_k(x)}\,\mathrm{d} x \right|
\\ && + \int_0^1 |f^{(\alpha)}(t)| \left[1_{[c,1]}(t) |\chi_\alpha^{(+)}(k)| +
1_{[0,c]}(t) |\chi_\alpha^{(-)}(k) | \right] \,\mathrm{d} t \\ & \le & C
b^{-\mu_\alpha(k)} \left(\sum_{r=1}^{\alpha-1} |f^{(r)}(c)|^2 +
\int_0^1 |f^{(\alpha)}(t)|^2 \,\mathrm{d} t \right)^{1/2} \\ & = & C
b^{-\mu_\alpha(k)} \|f\|_{H(K_{\alpha,c})},
\end{eqnarray*}
where the constant $C >0$ is independent of $k$ and $f$.
\end{proof}

One can show that the constant in Theorem~\ref{thm_embedding} can be chosen as $C_{3,\alpha} := \sqrt{\alpha} C_{1,\alpha}$, where $C_{1,\alpha}$ is given by \eqref{const_c1}.

The result can be generalized for tensor product spaces. Let $u
\subset \mathbb{N}$ be a finite set. For $\bsx_u = (x_i)_{i \in u},
\bsy_u = (y_i)_{i \in u} \in [0,1]^{|u|}$  let
\begin{equation*}
K_{\alpha,c,u}(\bsx_u,\bsy_u) = \prod_{i \in u}
K_{\alpha,c}(x_i,y_i).
\end{equation*}
This reproducing kernel defines a reproducing kernel Hilbert space
$H(K_{\alpha,c,u})$ with inner product $\langle \cdot,
\cdot \rangle_{\alpha,c,u}$ and corresponding norm $\|\cdot
\|_{\alpha,c,u}$.

For $\bsk_u = (k_i)_{i \in u} \in \mathbb{N}_0^{|u|}$ let
\begin{equation*}
\mu_{\alpha}(\bsk_u)  = \sum_{i \in u} \mu_{\alpha}(k_i).
\end{equation*}
We define the Walsh functions
\begin{equation*}
\wal_{\bsk_u}(\bsx_u) = \prod_{i\in u} \wal_{k_i}(x_i).
\end{equation*}

For $\alpha \ge 2$ we define the Walsh space $\mathcal{W}_{\alpha,u}$ as the
space of all Walsh series
\begin{equation*}
f(\bsx_u) = \sum_{\bsk_u \in \mathbb{N}^{|u|}_0} \widehat{f}(\bsk_u)
\wal_{\bsk_u}(\bsx_u)
\end{equation*}
with
\begin{equation*}
\|f\|_{\mathcal{W}_{\alpha,u}} = \sup_{\bsk_u \in \mathbb{N}^{|u|}_0}
|\widehat{f}(\bsk_u)| b^{\mu_\alpha(\bsk_u)} < \infty.
\end{equation*}

Using the representation $f(\bsx) = \langle f, K_{\alpha,c,u}
(\cdot, \bsx)
\rangle_{H(K_{\alpha,c,u})}$ one obtains a multidimensional Taylor series
with integral remainder. The $k_i$th Walsh coefficients of products
of $(x_i-c)^{r_i}$, $(x_i-t_i)_+^{\alpha-1}$ and
$(t_i-x_i)_+^{\alpha-1}$ can all be estimated by $C
b^{-\mu_\alpha(k_i)}$. Thus we obtain the following corollary.

\begin{corollary}
\label{Tau_Konvergenz}
Let $u \subset \mathbb{N}$ be a finite set. For $\alpha \ge 2$ the tensor product space
$H(K_{\alpha,c,u})$ is continuously embedded in
$\mathcal{W}_{\alpha,u}$. That is, there is a constant
$C_{4,\alpha,|u|} > 0$ such that for all $f \in
H(K_{\alpha,c,u})$ we have
\begin{equation*}
\|f\|_{\mathcal{W}_{\alpha,u}} \le C_{4,\alpha,|u|}
\|f\|_{H(K_{\alpha,c,u})}.
\end{equation*}
\end{corollary}

The constant $C_{4,\alpha,|u|}$ can be chosen as $C_{4,\alpha,|u|} = (C_{3,\alpha})^{|u|} = \alpha^{|u|/2} (C_{1,\alpha})^{|u|}$.

Consider now a reproducing kernel of the form
\begin{equation}
\label{alpha-kern}
K_{\alpha, \bsgamma}(\bsx,\bsy) = \sum_{u \subseteq [s]} \gamma_u K_{\alpha, c, u}(\bsx_u, \bsy_u),
\end{equation}
which defines the reproducing kernel Hilbert space $H(K_{\alpha, \bsgamma})$ with inner product $\langle \cdot, \cdot \rangle_{H(K_{\alpha, \bsgamma})}$ and corresponding norm $\|\cdot \|_{H(K_{\alpha, \bsgamma})}$. Further we define the Walsh space $\mathcal{W}_{\alpha, \bsgamma}$, $\bsgamma = (\gamma_u)_{u\subseteq [s]}$, as the space of all Walsh series
\begin{equation*}
f(\bsx) = \sum_{\bsk \in \mathbb{N}_0^s} \widehat{f}(\bsk) \wal_{\bsk}(\bsx),
\end{equation*}
with finite norm
\begin{equation*}
\|f\|_{\mathcal{W}_{\alpha,\widetilde{\bsgamma}}} = \max_{u \subseteq [s]} \widetilde{\gamma}_u^{-1} \|f_u\|_{\mathcal{W}_{\alpha,u}}
\end{equation*}
where $f_u = \langle f, K_{\alpha,c,u}\rangle_{H(K_{\alpha,\bsgamma})}$ is the projection of $f$ onto $H(K_{\alpha,c,u})$. Then we have
\begin{equation*}
\|f\|_{\mathcal{W}_{\alpha,\widetilde{\bsgamma}}} \le \left( \sum_{u\subseteq [s]} \gamma_u^{-1} \|f_u\|^2_{H(K_{\alpha,c,u})} \right)^{1/2} = \|f\|_{H(K_{\alpha, \bsgamma})},
\end{equation*}
where $\widetilde{\bsgamma} = (\widetilde{\gamma}_u)_{u \subseteq [s]}$ and $\widetilde{\gamma}_u = C_{4,\alpha,|u|} \sqrt{\gamma_u}$.

\subsection{Numerical integration}
\label{NUMINT}

Let $\alpha > 1$ be an integer. The worst-case integration error in $H(K_{\alpha,c,u})$ using a quasi-Monte Carlo algorithm $Q_P(f) = \frac{1}{|P|} \sum_{\bsx \in P} f(\bsx)$ based on the point set $P = \{\bsx_0,\ldots, \bsx_{N-1}\} \subset [0,1]^{u}$ is given by
\begin{equation*}
e(Q; H(K_{\alpha,c,u})) = \sup_{f \in
H(K_{\alpha,c,u}), \|f\|_{H(K_{\alpha,c,u})} \le 1}
\left|\int_{[0,1]^u} f(\bsx_u) \,\mathrm{d} \bsx_u - \frac{1}{N}
\sum_{n=0}^{N-1} f(\bsx_n) \right|.
\end{equation*}

Since the reproducing kernel Hilbert space
$H(K_{\alpha,c,u})$ is continuously embedded in
$\mathcal{W}_{\alpha,u}$, the results on numerical integration of
\cite{BDGP11} in $\mathcal{W}_{\alpha,u}$ apply. From
\cite[Theorem~3.1]{BDGP11} we obtain the following result which will be used in the changing dimension algorithm.

\begin{proposition}\label{prop1}
Let $b$ be a prime number, $m \ge 1$ and $\alpha \ge 2$ be integers. Then a higher order polynomial lattice point set $\cS_{p,m,\alpha}(\bsq)$ with modulus $p$ of degree $\alpha m$ constructed over the finite field $\mathbb{Z}_b$ of order $b$
and generating vector $\bsg \in \mathbb{Z}_b^{|u|}$ can be constructed component-by-component such that the quasi-Monte Carlo rule $Q_{\bsg,p}$ using the quadrature points $\cS_{p,m,\alpha m}(\bsq)$ satisfies
\begin{equation}
\label{estimate-for-CDA}
e(Q_{\bsg,p}; H(K_{\alpha,c,u})) \le \frac{1}{b^{\tau m}}
C_{b,\alpha,1/\tau}^{|u| \tau} \quad \mbox{for all } 1 \le \tau <
\alpha.
\end{equation}
The constant here is given by
\begin{equation*} C_{b,\alpha,1/\tau} := 1 + C_{3,\alpha} \left(\widetilde{C}_{b,\alpha,1/\tau} + \frac{(b-1)^\alpha}{b^{\alpha/\tau} -b} \prod_{j=1}^{\alpha -1} \frac{1}{b^{j/\tau}-1} \right),
\end{equation*}
where $C_{3,\alpha}$ is as in Section \ref{EMBEDDING} and
\begin{equation*}
\widetilde{C}_{b,\alpha,1/\tau} := \left\{
\begin{array}{ll}
\alpha-1 & \mbox{ if } \tau =1,\\
\frac{(b-1)((b-1)^{\alpha-1} - (b^{1/\tau} -1)^{\alpha-1})}{(b-b^{1/\tau})(b^{1/\tau}-1)^{\alpha-1}}& \mbox{ if } \tau > 1 .
\end{array}\right.
\end{equation*}
\end{proposition}

Note that one does not require a random digital shift of the polynomial lattice point set in Proposition~\ref{prop1} due to the embedding of the function space $H(K_{\alpha,c,u})$ in the Walsh space. This random digital shift is however required for $\alpha = 1$ to get a corresponding result (which is not covered in Proposition~\ref{prop1}).

The construction cost of the component-by-component algorithm is of $O(|u| N^\alpha \alpha \log N)$ operations using $O(N^\alpha)$ memory (where $N = b^m$ is the number of points), see \cite{BDLNP12}.

Consider now a reproducing kernel of the form (\ref{alpha-kern}).
%\begin{equation*}
%K_{\alpha,\bsgamma}(\bsx,\bsy) = \sum_{u \subseteq [s]} \gamma_u K_{\alpha,c,u}(\bsx_u,\bsy_u),
%\end{equation*}
%where the sum is over some finite subsets $u$ of $\mathbb{N}$  and $\gamma_u > 0$.
For functions $f \in H(K_{\alpha,\bsgamma})$ with anchored decomposition $f= \sum_{u\subseteq [s]} f_u = \sum_{u \subseteq [s]} \sum_{\bsk_u \in \mathbb{N}_0^{|u|}} \widehat{f}_u(\bsk_u) \wal_{\bsk_u}$
%, which is given by (cf. \cite{BDGP11}),
we have
\begin{align*}
\left|\int_{[0,1]^s} f(\bsx)\,\rd \bsx - \frac{1}{b^m} \sum_{n=0}^{b^m-1} f(\bsx_n) \right| \le & \|f\|_{\mathcal{W}_{\alpha,\widetilde{\bsgamma}}} \sum_{\emptyset \neq u \subseteq [s]} \widetilde{\gamma}_u \frac{1}{b^m} \sum_{n=0}^{b^m-1} \sum_{\bsk_u \in \mathbb{N}_0^{|u|} \setminus \{\bszero\}} b^{-\mu_\alpha(\bsk_u)} \wal_{\bsk_u}(\bsx_{n,u}) \\ \le & \|f\|_{\mathcal{W}_{\alpha,\widetilde{\bsgamma}}} \sum_{\emptyset \neq u \subseteq [s]}  \frac{1}{b^m} \sum_{n=0}^{b^m-1} \widetilde{\gamma}'_u \sum_{\bsk_u \in \mathbb{N}^{|u|}} b^{-\mu_\alpha(\bsk_u)} \wal_{\bsk_u}(\bsx_{n,u}) \\ \le & \|f\|_{H(K_{\alpha,\bsgamma})} \sum_{\emptyset \neq u \subseteq [s]}  \frac{1}{b^m} \sum_{n=0}^{b^m-1} \widetilde{\gamma}'_u \sum_{\bsk_u \in \mathbb{N}^{|u|}} b^{-\mu_\alpha(\bsk_u)} \wal_{\bsk_u}(\bsx_{n,u}),
\end{align*}
where $\widetilde{\gamma}_u = C_{3,\alpha}^{|u|} \sqrt{\gamma_u}$ and
$\widetilde{\gamma}'_u = \sum_{u \subseteq v\subseteq [s]} \widetilde{\gamma}_v$ (note that $\frac{1}{b^m} \sum_{n=0}^{b^m-1} \wal_{\bsk_u}(\bsx_{n,u})$ only takes on the values $0$ or $1$). Let $\gamma'_u = \sum_{v \supseteq u} \gamma_v$. Using a slight generalization of \cite[Theorem~3.1]{BDGP11} we obtain that a higher order polynomial
lattice point set $\cS_{p,m,\alpha m}(\bsq)$ with modulus $p$ of degree $\alpha m$
and generating vector $\bsg$ can be constructed
component-by-component  such that the quasi-Monte Carlo rule $Q_{\bsg,p}$ using the quadrature points $\cS_{p,m,\alpha m}(\bsq)$ satisfies
\begin{align*}%\label{36'}
e(Q_{\bsg,p};H(K_{\alpha,\bsgamma})) \le & \frac{1}{b^{\tau m}}
\left( \sum_{u \subseteq [s]} (\widetilde{\gamma}'_u)^{1/(\tau)} (2 C_{b,\alpha,1/\tau})^{|u|}
\right)^\tau \nonumber \\ \le & \frac{1}{b^{\tau m}}
\left( \sum_{u \subseteq v \subseteq [s]} \gamma_v^{1/(2\tau)} C_{3,\alpha}^{|v|/\tau} (2 C_{b,\alpha,1/\tau})^{|u|}
\right)^\tau \nonumber \\ = & \frac{1}{b^{\tau m}}
\left( \sum_{v \subseteq [s]} \gamma_v^{1/(2\tau)} C_{3,\alpha}^{|v|/\tau} (1 + 2 C_{b,\alpha,1/\tau})^{|v|} \right)^\tau \quad \mbox{for all } 1 \le \tau < \alpha.
\end{align*}
Note that the construction above is explicit, however, the range of $\tau$ is restricted to $1 \le \tau < \alpha$. In the following we therefore consider the range $1/2 \le \tau < 1$. If one chooses $1/2 \le \tau < 1$, then one can use the construction of polynomial lattice rules from \cite{DKPS} to obtain the result that there exists a digital shift $\bssigma \in [0,1)^s$ such that
\begin{align}\label{eq_error_tauhalb}
e(Q_{\bsg,p}(\bssigma);H(K_{1,\bsgamma})) \le &  \frac{1}{b^{\tau m}}
\left( \sum_{u \subseteq [s]} \gamma_u^{1/(2\tau)} (C'_{\tau})^{|u|}
\right)^\tau \quad \mbox{for all } 1/2 \le \tau < 1,
\end{align}
for some suitable constant $C'_\tau > 0$ independent of $s$ and $m$. Note that the space $H(K_{\alpha,\bsgamma})$ is continuously embedded in the space $H(K_{\alpha-1, \bsgamma'})$, where $\bsgamma' = (2^{|u|} \gamma_u)_{u \subseteq [s]}$. This follows from the tensor product structure of the reproducing kernel Hilbert spaces $H(K_{\alpha,c,u})$ and
\begin{equation*}
\frac{1}{2} \int_0^1 |f^{(\alpha-1)}(x)|^2 \rd x \le |f^{(\alpha-1)}(c)|^2 + \int_0^1 |f^{(\alpha)}(x)|^2 \rd x,
\end{equation*}
which in turn follows from
\begin{equation*}
f^{(\alpha-1)}(x) = f^{(\alpha-1)}(c) + \int_c^x f^{(\alpha)}(t) \rd t,
\end{equation*}
for $x \ge c$ and an analogous expression for $x < c$. Thus functions in $H(K_{\alpha,\bsgamma})$ are also in $H(K_{1,\bsgamma''})$, where $\bsgamma'' = (2^{(\alpha-1) |u|} \gamma_u)_{u \subseteq [s]}$. Therefore \eqref{eq_error_tauhalb} applies for functions in $H(K_{\alpha,\bsgamma})$ where one replaces the constant $C'_\tau$ with $2^{\frac{\alpha -1}{2\tau}} C'_\tau$.

Note that we have
\begin{align*}
& [e(Q_{\bsg,p}(\bssigma); H(K_{\alpha,\bsgamma}))]^2 \\  = & \sum_{u \subseteq [s-1]} \gamma_u \left[ e \left( \left( Q_{\bsg, p}(\bssigma) \right)_{u}; H(K_{\alpha,c,u}) \right) \right]^2 + \sum_{s \in u \subseteq [s]} \gamma_u \left[ e \left( \left(
Q_{\bsg, p}(\bssigma) \right)_{u}; H(K_{\alpha,c,u}) \right) \right]^2.
\end{align*}
In the component-by-component algorithm one updates the components $g_j$ of $\bsg$ inductively. The first sum over all subsets $u \subseteq [s-1]$ does not depend on the last component and is therefore fixed when updating $g_s$. The component-by-component algorithm then minimizes the second sum over all subsets $s \in u \subseteq [s]$ and this sum is then shown to satisfy the bound
\begin{equation}\label{genau-richtig}
 \sum_{s \in u \subseteq [s]} \gamma_u \left[ e \left( \left(
Q_{\bsg, p}(\bssigma) \right)_{u}; H(K_{\alpha,c,u}) \right) \right]^2 \le  \frac{1}{b^{2 \tau m}} \left( \sum_{s \in u \subseteq [s]} \gamma_u^{1/(2\tau)} (C'_{\tau})^{|u|}
\right)^{2\tau }.
\end{equation}
This implies that for any $1/2 \le \tau < \alpha$ there is a polynomial lattice rule together with a digital shift $\bssigma$ such that
%\begin{equation}
%\sum_{s \in u \subseteq [s]} \gamma_{u} \left[ e \left( \left(
%Q_{\bsg, p}(\bssigma) \right)_{u}; H(K_{\alpha,c,u}) \right) \right]^2
%\le  \frac{1}{b^{2 \tau m}}
%\left( \sum_{s \in u \subseteq [s]} \gamma_u^{1/(2\tau)} (C'_{\tau})^{|u|}
%\right)^{2\tau }.
%\end{equation}
(\ref{genau-richtig}) holds. For $\tau \ge 1$ one can choose the digital shift $\bssigma = \bszero$.

Such polynomial lattice rules can be constructed using a component-by-component algorithm as shown in \cite{DKPS} for $1/2 \le \tau < \alpha = 1$ and in \cite{BDGP11} for $1 \le \tau < \alpha$.

\subsection{Results for product and order-dependent weights}
\label{HOC-POD}

\subsubsection{Nested subspace sampling}
\label{HOC-POD-NEST}

Let $\$(k)=O(k^s)$ for some $s\ge 0$. Let $\bsgamma = (\gamma_u)_{u\in \U}$ be POD weights that satisfy the assumptions of
Corollary \ref{Corollary} and have $\decay_{\bsgamma,1} >1$. For $k\in\N$ and the set %$\{u_j \,|\, j\in V_k\}$ of finite subsets of $\NN$
$v_k = v_k^{(2)} = [L_k]$, see Section \ref{MLA},
we may apply estimate (\ref{genau-richtig}) to see that our assumption (\ref{assumption3.9}) holds.\footnote{Recall that polynomial lattice rules consist
of $n$ points, where $n$ is a power of a prime $b$. If required to construct a quadrature rule consisting of $n$ points, $n\in\N$ arbitrary, we generate a polynomial lattice rules consisting of $b^m$ points, $b^m\le n <b^{m+1}$, and set the quadrature weights corresponding to the ``missing'' $n-b^m$ points simply to zero.}
Thus the estimates from Theorem \ref{Theorem6} for $p^{\nes}$ can be established
by multilevel algorithms using as
building blocks the polynomial lattice rules explained above.
Due to the fact that $e(n;H(K_{\alpha,c})) = \Omega(n^{-\alpha})$ and our lower bound (\ref{neslowboupod})  we get, in particular, the following result.

\begin{corollary}\label{Cor_Theorem6}
Let $\$(k)=\Theta(k^s)$ for some $s\ge 0$. Let $\bsgamma = (\gamma_u)_{u\in \U}$ be POD weights that satisfy the assumptions of
Corollary \ref{Corollary}. Let $\alpha \ge 1$ be an integer.
Then our quasi-Monte Carlo multilevel algorithms $Q^{\ML}_m$,
defined as in (\ref{multilevel-algo}) with polynomial
lattice rules as in Section \ref{NUMINT} as quadrature rules $Q_{v_k}$,  establish the following result: The infinite-dimensional integration problem is strongly tractable in the nested subspace sampling model.

In the case where $s\ge (2\alpha-1)/2\alpha$ we obtain
\begin{equation}
p^{\nes} = \max \left\{ \frac{1}{\alpha}, \frac{2s}{\decay_{\bsgamma,1} - 1} \right\}.
\end{equation}

In the case where $0 \le s < (2\alpha-1)/2\alpha$, we obtain for
\begin{itemize}
\item[] $\decay_{\bsgamma,1} \ge 2\alpha$:
\begin{equation*}
p^{\nes} = \frac{1}{\alpha},
\end{equation*}
\item[] $2\alpha > \decay_{\bsgamma,1} > 1/(1-s)$:
\begin{equation*}
\max \left\{ \frac{1}{\alpha}, \frac{2s}{\decay_{\bsgamma,1}-1} \right\} \le p^{\nes} \le \frac{2}{\decay_{\bsgamma,1}},
\end{equation*}
\item[] $1/(1-s) \ge \decay_{\bsgamma,1} >1$:
\begin{equation*}
p^{\nes} = \frac{2s}{\decay_{\bsgamma,1}-1}.
\end{equation*}
\end{itemize}
\end{corollary}

\subsubsection{Unrestricted subspace sampling}
\label{USSHIGHERORDER}

If the cost function satisfies $\$(k) = O(k^s)$ for $0\le s\le 1$, we can use
the quasi-Monte Carlo multilevel algorithms from Section \ref{HOC-POD-NEST} and achieve the
same result as in Corollary~\ref{Cor_Theorem6}.
If $\$(k) = \Omega(k^s)$ for $s \ge 1$, we can use changing dimension algorithms
as in (\ref{def-CD}) with polynomial lattices rules as in Proposition \ref{prop1}.
Due to Corollary \ref{Cor_Theorem6} and Theorem \ref{Thm_hoc_cda} these QMC multilevel
and changing dimension algorithms lead to
the following result.

\begin{corollary}
Let $\$(k) =\Theta(k^s)$ for some $s\ge 0$.
Let $\bsgamma = (\gamma_u)_{u\in \U}$ be POD weights that satisfy the assumptions of
Corollary \ref{Corollary}. Let $\alpha > 1$ be an integer.

If $s\ge (2\alpha - 1)/2\alpha$, then
the infinite-dimensional integration problem is strongly tractable with exponent
\begin{equation*}
p^{\unr} = \max \left\{ \frac{1}{\alpha},
\frac{2 \min\{1,s\}}{\decay_{\bsgamma,1}-1} \right\}.
\end{equation*}

If $s < (2\alpha - 1)/2\alpha$, then
the infinite-dimensional integration problem is strongly tractable and
$p^{\unr}$ satisfies the same relations as $p^{\nes}$ in Corollary
\ref{Cor_Theorem6}.
\end{corollary}

\subsection{Results for weights with finite algorithmic dimension}

Let us briefly mention the results that our quasi-Monte Carlo multilevel and
changing dimension algorithms achieve in the case of weights with finite algorithmic
dimension.

We now show how quadrature rules which satisfy \eqref{assuconv} can be constructed explicitly.  Choose $\bst^{(i,n)}$ in \eqref{algobaustein-FAD} to be the first $n$ points of a $(t,\alpha,d)$-sequence as constructed in \cite{D08}. The weights $a_i^{(n)}$ can be chosen in the following way: Let $m$ be an integer such that $b^m \le n < b^{m+1}$. Then set $a_i^{(n)} = b^{-m}$ for $1 \le i \le b^m$ and $0$ for $b^m < i \le n$. Then \cite[Theorem~5.4]{D08} together with Corollary~\ref{Tau_Konvergenz} implies that this quadrature rule satisfies \eqref{assuconv}. In the following two theorems let $Q_n$ denote the higher order quasi-Monte Carlo rule as described in this paragraph.

\subsubsection{Nested subspace sampling}

Due to Theorem \ref{UppBouFAD} we obtain the following corollary.

\begin{corollary}\label{Vorletztes_Korollar}
%\label{UppBouFAD}
Let $\$(k) = O(k^s)$ for some $s\ge 0$. Let the weights $\bsgamma$ have finite
algorithmic dimension, and let $\decay_{\bsgamma} >1$. Let $\alpha \ge 1$ be an integer. Then for all $n \in \N$ the higher order quasi-Monte Carlo rules $Q_n$ satisfies (\ref{assuconv}). For $k=1,2,\ldots$,
let $Q_{v_k} = Q^{\mathcal{W}}_{n,v_k}$ be as in (\ref{algoW}).
Then the multilevel algorithms $Q^{\ML}_m$, defined as in (\ref{multilevel-algo}),
establish the following result:
The exponent of strong tractability in the nested subspace sampling
model satisfies
\begin{equation*}
%\label{uppboufad}
p^{\nes} \le \max \left\{ \frac{1}{\alpha}, \frac{2s}{\decay_{\bsgamma} - 1} \right\}.
\end{equation*}
\end{corollary}

The lower bound (\ref{neslowboufad}) on $p^{\nes}$ shows that the upper bound in
Corollary \ref{Vorletztes_Korollar} is sharp for finite-intersection weights.

\subsubsection{Unrestricted subspace sampling}

In the unrestricted subspace sampling setting we use for $\$(k)=O(k^s)$ and $s\le 1$
multilevel algorithms $Q^{\ML}_m$ as in Corollary \ref{Vorletztes_Korollar}, and for
$s>1$ changing dimension algorithms, see Section \ref{SUBSEC_4.4.2},
that rely on the higher order quasi-Monte Carlo rules $Q_n$ described above.
This results in the following corollary.

\begin{corollary}\label{Letztes_Korollar}
Let $\$(k) = O(k^s)$ for some $s\ge 0$. Let the weights $\bsgamma$ have finite
algorithmic dimension, and let $\decay_{\bsgamma} >1$.
Let $\alpha \ge 1$ be an integer. Then the exponent of strong tractability in the
unrestricted subspace sampling model satisfies
\begin{equation*}
%\label{uppboufad}
p^{\unr} \le \max \left\{ \frac{1}{\alpha}, \frac{2\min\{1,s\}}{\decay_{\bsgamma} - 1} \right\}.
\end{equation*}
\end{corollary}

The lower bound (\ref{neslowboufad}) on $p^{\unr}$ shows that the upper bound in
Corollary \ref{Letztes_Korollar} is sharp for finite-intersection weights.

\section{Appendix}
\label{APPENDIX}

Here we provide a detailed proof of Lemma \ref{lem_pod_example}.

\begin{lemma}\label{lem_pod_gen}
Let $r > 1$ be a real number and define the POD weights $\gamma_u = \Gamma_{|u|} \prod_{j\in u} j^{-r}$ for $u \in \U$. Then there is a constant $c_r > 0$ such that
\begin{equation}
 \label{hurwitz1}
\sum_{u \in \U} \gamma_u \ge \Gamma_0 + c_r \sum_{k=1}^\infty \frac{\Gamma_k}{(k!)^{2 \lceil r/2 \rceil}} k^{-\lceil r/2 \rceil} \left(\frac{\pi}{2 \lceil r/2 \rceil \sin \pi / (2 \lceil r/2 \rceil) } \right)^{rk}.
\end{equation}
If $r \ge 2$, then there is a constant $C_r > 0$ such that
\begin{equation}
\label{hurwitz2}
\sum_{u \in \U} \gamma_u \le \Gamma_0 + C_r \sum_{k=1}^\infty \frac{\Gamma_k}{(k!)^r} k^{-r/2} \left(\frac{\pi}{2 \lfloor r/2 \rfloor \sin \pi / (2 \lfloor r/2 \rfloor) }\right)^{rk}.
\end{equation}
\end{lemma}

Note that $\sin x < x$ for $x > 0$, thus $\sin \pi/r < \pi/r$, which implies
\begin{equation*}
1 < \frac{\pi}{r \sin \pi/r}.
\end{equation*}

\begin{proof}
We have
\begin{equation*}
\sum_{u \in \U} \gamma_u = \sum_{k=0}^\infty \Gamma_{k} \sum_{\satop{u \in \U}{|u| = k}} \prod_{j\in u} j^{-r} = \Gamma_0 + \sum_{k=1}^\infty \Gamma_{k} \sum_{1 \le j_1 < j_2 < \cdots < j_k} \prod_{i=1}^k  j_i^{-r} = \Gamma_0 + \sum_{k=1}^\infty \Gamma_k \zeta(\underbrace{r,\ldots, r}_{k \mbox{ times}}),
\end{equation*}
where $\zeta(\underbrace{r,\ldots, r}_{k \mbox{ times}})$ is the multiple Hurwitz zeta function.

The general behavior of the multiple Hurwitz zeta function is given in \cite[Eq. (48)]{BBB97}. From \cite[p. 8]{BBB97} it is known that if $r \ge 2$ is an even integer, then
\begin{equation*}
\zeta(\underbrace{r,\ldots, r}_{k \mbox{ times}}) =  \frac{r (2\pi)^{rk}}{(rk  + r/2)!} \left(\frac{1}{2 \sin \pi/r} \right)^{rk+r/2} \left(1 + \sum_{j=2}^{N_{r}} R_{r,j}^{rk+r/2} \right),
\end{equation*}
where $R_{r,j}$ are some numbers with $|R_{r,j}| < 1$ and $N_{r}$ is a positive integer satisfying $N_r < 2^{r/2}/r$. From Stirling's formula we obtain
\begin{equation*}
\frac{(k!)^r}{(rk)!} \asymp_k \frac{\sqrt{2\pi}}{\mathrm{e}} \frac{k^{kr} \mathrm{e}^{-rk}}{(rk)^{rk} \mathrm{e}^{-rk}} =  \frac{\sqrt{2\pi}}{\mathrm{e}} r^{-rk},
\end{equation*}
where $f(k) \asymp_k g(k)$ means that there are constants $C,c> 0$ independent of $k$ such that $c g(k) \le f(k) \le C g(k)$. Thus
\begin{align*}
(k!)^r \zeta(\underbrace{r,\ldots, r}_{k \mbox{ times}}) & =  \frac{(k!)^r r (2\pi)^{rk}}{(rk  + r/2)!} \left(\frac{1}{2 \sin \pi/r} \right)^{rk+r/2} \left(1 + \sum_{j=2}^{N_{r}} R_{r,j}^{rk+r/2} \right) \\ & \asymp_k \frac{\sqrt{2\pi} r}{\mathrm{e}} \left(\frac{1}{2\sin \pi/r} \right)^{r/2} \frac{1}{(rk+r/2)^{r/2}} \left(\frac{\pi}{r \sin \pi/r} \right)^{rk} \left(1 + \sum_{j=2}^{N_{r}} R_{r,j}^{rk+r/2} \right) \\ & \asymp_k \frac{1}{k^{r/2}} \left(\frac{\pi}{r \sin \pi/r} \right)^{rk}.
\end{align*}

Thus, for any fixed positive even integer $r$ we have
\begin{equation*}
\sum_{k=1}^\infty \Gamma_k \zeta(\underbrace{r,\ldots, r}_{k \mbox{ times}}) = \sum_{k=1}^\infty \frac{\Gamma_k}{(k!)^r} (k!)^r \zeta(\underbrace{r,\ldots, r}_{k \mbox{ times}}) \asymp \sum_{k=1}^\infty \frac{\Gamma_k}{(k!)^r} k^{-r/2} \left(\frac{\pi}{r \sin \pi/r}\right)^{rk}.
\end{equation*}
Therefore (\ref{hurwitz1}) follows since decreasing $r$ only increases the sum $\sum_{u \in \U} \gamma_u$ and the result holds for all even integers $r \ge 2$ as shown above.

Now assume that $r \ge 2$. For $1/r < \lambda \le 1$ we have by Jensen's inequality that
\begin{equation*}
[\zeta(r,\ldots, r)]^\lambda = \left[\sum_{1\le j_1 < \cdots < j_k} \prod_{i=1}^k j_i^{-r} \right]^\lambda \le \sum_{1 \le j_1 < \cdots < j_k} \prod_{i=1}^k j_i^{-r \lambda} = \zeta(r\lambda,\ldots, r\lambda).
\end{equation*}
Choose $1/r < \lambda \le 1$ such that $\lambda r$ is the largest even integer smaller or equal  than $r$. Then
\begin{equation*}
(k!)^r \zeta(r,\ldots, r) \le \left[(k!)^{\lambda r} \zeta(r\lambda,\ldots, r\lambda)\right]^{1/\lambda} \le C_r \frac{1}{k^{r/2}} \left(\frac{\pi}{\lambda r \sin \pi/ (\lambda r)} \right)^{rk},
\end{equation*}
for some constant $C_r > 0$. Thus
\begin{equation*}
\sum_{u \in \U} \gamma_u \le \Gamma_0 + C_r \sum_{k=1}^\infty \frac{\Gamma_k}{(k!)^r} k^{-r/2} \left(\frac{\pi}{\lambda r \sin \pi/ (\lambda r)} \right)^{rk},
\end{equation*}
from which (\ref{hurwitz2}) follows.
\end{proof}

\begin{corollary}\label{cor_pod_criteria}
Let $\bsgamma = (\gamma_u)_{u \in \U}$ be POD weights with $\gamma_u = \Gamma_{|u|} \prod_{j\in u} \gamma_j$. Let
%\begin{equation*}
$p^*:=\decay_{\bsgamma,1}
%p^\ast = \sup\left\{p \in \mathbb{R}: \sum_{j=1}^\infty \widetilde{\gamma}_j^{1/p} %< \infty\right\}
< \infty$.
%\end{equation*}
%Then
%\begin{equation*}
%\mathrm{decay}_{\bsgamma,\infty} = p^\ast.
%\end{equation*}
Further let $c, c_0 > 0$ be constants such that
\begin{equation*}
c_0 j^{-p^\ast} \le \gamma_j \le c j^{-p^\ast} \quad \mbox{for all } j \ge 1.
\end{equation*}
If for some $q \le p^\ast/2$ we have
\begin{align}\label{eq_pod1}
\sum_{k=1}^\infty \frac{c^{k/q} \Gamma_k^{1/q}}{(k!)^{p^\ast /q}} k^{-p^\ast/ (2q)} \left(\frac{\pi}{2 \lfloor p^\ast/(2q) \rfloor \sin \pi/(2 \lfloor p^\ast / (2q) \rfloor} \right)^{k p^\ast/q} < \infty,
\end{align}
then $\mathrm{decay}_{\bsgamma,\infty} \ge q$.

On the other hand, if for $q < p^\ast$ we have
\begin{equation}\label{eq_pod2}
\sum_{k=1}^\infty \frac{c_0^{k/q} \Gamma_k^{1/q}}{(k!)^{2 \lceil p^\ast/(2q) \rceil}} k^{-\lceil p^\ast/(2q) \rceil} \left(\frac{\pi}{2 \lceil p^\ast/(2q) \rceil \sin \pi / (2 \lceil p^\ast/(2q) \rceil) } \right)^{k p^\ast /q} = \infty,
\end{equation}
then $\mathrm{decay}_{\bsgamma,\infty} \le q$.
\end{corollary}

\begin{proof}
We have
\begin{equation*}
\mathrm{decay}_{\bsgamma,\infty} = \sup\left\{q \in \mathbb{R}: \sum_{u \in \U} \gamma_u^{1/q} < \infty \right\}.
\end{equation*}
Thus we have for some $q \le p^\ast/2$
\begin{align*}
\sum_{u \in \U} \gamma_u^{1/q} \le \Gamma_0^{1/q} + C_{p^*/q} \sum_{k=1}^\infty  \frac{c^{k/q} \Gamma_k^{1/q}}{(k!)^{p^\ast /q}} k^{-p^\ast/ (2q)} \left(\frac{\pi}{2 \lfloor p^\ast/(2q) \rfloor \sin \pi/(2 \lfloor p^\ast / (2q) \rfloor} \right)^{k p^\ast/q}
\end{align*}
that the right hand side is finite, then $\mathrm{decay}_{\bsgamma,\infty} \ge q$.

On the other hand, for $q < p^\ast$ we have
\begin{equation*}
\sum_{u \in \U} \gamma_u^{1/q} \ge \Gamma_0^{1/q} + c_{p^*/q} \sum_{k=1}^\infty \frac{c_0^{k/q} \Gamma_k^{1/q}}{(k!)^{2 \lceil p^\ast/(2q) \rceil}} k^{-\lceil p^\ast/(2q) \rceil} \left(\frac{\pi}{2 \lceil p^\ast/(2q) \rceil \sin \pi / (2 \lceil p^\ast/(2q) \rceil) } \right)^{p^\ast k/q}.
\end{equation*}
If the right hand side is infinite for some $q < p^\ast$, then $\mathrm{decay}_{\bsgamma,\infty} \le q$.
\end{proof}

We suspect that the condition $q \le p^\ast/2$ in the above corollary can be replaced by $q \le p^\ast$.

The corollary above allows us to construct an example of POD weights where
\begin{equation*}
1 \le \mathrm{decay}_{\bsgamma,\infty} < \mathrm{decay}_{\bsgamma,1}.
\end{equation*}
For instance, let $\gamma_j = j^{-p^\ast}$. Thus $\mathrm{decay}_{\bsgamma,1} = p^\ast$ and $c_0 = c = 1$ in the above corollary.
Let $q^\ast$ be such that $p^\ast/(2q^\ast) \in \mathbb{N}$.
For $k \in \mathbb{N}_0$ let
\begin{equation*}
\Gamma_k = (k!)^{p^\ast} k^{p^\ast/2-q^\ast} \left(\frac{(p^\ast/q^\ast) \sin (q^*\pi/p^\ast)}{\pi} \right)^{k p^\ast}.
\end{equation*}
Then we have for $q=q^*$ that \eqref{eq_pod2} is of the same form as \eqref{eq_pod1}, which is
\begin{equation}\label{eq_pod_example}
\sum_{k=1}^\infty \frac{\Gamma_k^{1/q}}{(k!)^{p^\ast/q}} k^{-p^\ast/(2q)} \left(\frac{\pi}{2 \lfloor p^\ast/(2q) \rfloor  \sin \pi / (2 \lfloor p^\ast/(2q) \rfloor) } \right)^{k p^\ast/q} = \sum^\infty_{k=1} k^{-1}= \infty.
\end{equation}
Due to (\ref{eq_pod2}) we have $\decay_{\bsgamma,\infty} \le q^*$.

Let now $q<q^*$ such that $\lfloor p^*/2q \rfloor = p^*/2q^*$.
For this $q$ the left hand side of (\ref{eq_pod1}) is
\begin{equation*}
 \sum^\infty_{k=1} \frac{\Gamma_k^{1/q}}{(k!)^{p^*/q}}
k^{-p^*/2q} \left( \frac{\pi}{(p^*/q^*) \sin(q^*\pi/p^*)} \right)^{kp^*/q}
= \sum^\infty_{k=1} k^{-q^*/q} <\infty.
\end{equation*}
Thus (\ref{eq_pod1}) gives us $\decay_{\bsgamma,\infty} \ge q$.
%Then \eqref{eq_pod_example} is
%\begin{equation*}
%\sum_{k=1}^\infty \frac{\Gamma_k^{1/q}}{(k!)^{p^\ast/q}} k^{-p^\ast/(2q)} %\left(\frac{\pi}{2 \lfloor p^\ast/(2q^\ast) \rfloor  \sin \pi / (2 \lfloor %p^\ast/(2q^\ast) \rfloor) } \right)^{k p^\ast/q} = \sum_{k=1}^\infty k^{-q^\ast/q},
%\end{equation*}
%which is finite for $q < q^\ast$ but infinite for $q = q^\ast$. Thus we have shown

Together with Lemma \ref{Lemma3.8} this establishes Lemma \ref{lem_pod_example}.

\subsection*{Acknowledgment}
Both authors want to thank Michael Griebel for suggesting them
to study algorithms for infinite-dimensional integration of
higher order convergence. We are grateful for the opportunity to work at the Hausdorff Institute in Bonn where the work on this paper was initiated.

Josef Dick is supported by an ARC Queen Elizabeth II Fellowship.

Michael Gnewuch was supported by the German Science Foundation DFG under grant GN 91/3-1 and by the Australian Research Council ARC.

\end{document}